\documentclass[11pt]{amsart}%[elsart, a4paper]{article}
\usepackage{amsaddr} %address in front.
\usepackage{amsmath, color}
\usepackage{amsfonts}
\usepackage{latexsym}
\usepackage{ifpdf}
\usepackage{graphicx,amssymb,lineno}
\usepackage{amssymb}
\usepackage{mathrsfs}
\usepackage{cite}
\usepackage{extarrows}
\usepackage{enumerate}
\usepackage{ulem}
\usepackage[all]{xy}
\allowdisplaybreaks[4]
\numberwithin{equation}{section}

\newtheorem{theorem}{{\bf Theorem}}[section]
\newtheorem{proposition}{{\bf Proposition}}[section]
\newtheorem{definition}{{\bf Definition}}[section]
\newtheorem{lemma}{{\bf Lemma}}[section]
\newtheorem{corollary}{{\bf Corollary}}[section]
\newtheorem{remark}{{\bf Remark}}[section]

\begin{document}
\title[Correlation functions of charged free bosons and fermions]
{\Large Correlation functions of charged free boson and fermion systems}
%$bc\beta\gamma$ System}
\author{Naihuan Jing${^{1,2}}$, Zhijun Li$^{\dagger 1}$, Tommy Wuxing Cai$^{3}$}
\address{$^1$School of Mathematics, South China University of Technology, Guangzhou 510640, China}%Department of Mathematics,
\address{$^2$Department of Mathematics, North Carolina State University, Raleigh, NC 27695, USA}
\email{jing@ncsu.edu, zhijun1010@163.com, cait@myumanitoba.ca}
\address{$^3$Department of Mathematics, University of Manitoba, Winnipeg, MB R3T 2N2 Canada}
%\email{cait@myumanitoba.ca}%{caiwx@math.scut.edu.cn}
\thanks{{\scriptsize
\hskip -0.6 true cm MSC (2010): Primary: 17B37; Secondary: 58A17, 15A75, 15B33, 15A15, 05E05.
\newline $^\dag$ Corresponding author: zhijun1010@163.com
\newline Supported by NSFC (grant no. 11531004), Simons Foundation (grant no. 523868).
}}

\maketitle
\begin{abstract} Using the idea of the quantum inverse scattering method, we
introduce the operators $\mathbf{B}(x), \mathbf{C}(x)$ and $\mathbf{\tilde{B}}(x), \mathbf{\tilde{C}}(x)$ corresponding to
the off-diagonal entries of the monodromy matrix $T$ for the phase model and $i$-boson model in terms of
bc fermions and neutral fermions respectively, thus giving alternative treatment of the KP and BKP hierarchies.
 We also introduce analogous operators $\mathbf{B}^{*}(x)$ and $\mathbf{C}^{*}(x)$ for the
charged free boson system and show that they are in complete analogy to those of $bc$ fermionic fields.
It is proved that the correlation function $\langle 0|\mathbf{C}(x_N)\cdots\mathbf{C}(x_1)\mathbf{B}(y_1)\cdots $ $\mathbf{B}(y_N)|0\rangle$ in the $bc$ fermionic fields is the inverse of the correlation function $\langle 0|\mathbf{C}^{*}(x_N)\cdots\mathbf{C}^{*}(x_1)\mathbf{B}^{*}(y_1)\cdots \mathbf{B}^{*}(y_N)|0\rangle$ in the charged free bosons.
\end{abstract}

\section{introduction}

The KP hierarchy is one of the fundamental examples of integrable systems that can be solved by many methods. Besides being
most interesting differential equations, the KP hierarchy and the KdV equation
are also one of the successful stories in Lie theory and mathematical physics. In a series
of papers \cite{DKM2, DJKM3, DJKM5, DJKM6, DJKM7},
the Kyoto school used infinite dimensional Lie algebras and the boson-fermion correspondence to study the KP and their
associated systems. They have shown that, among many things, the KP tau functions are expressed in terms of the Schur functions
from algebraic combinatorics
(see also \cite{OR2003}). Similarly the tau functions of the BKP hierarchy are determined by Schur's Q-functions \cite{You1990, Kv1990, Or2003, NO2017}.

The quantum inverse scattering method (QISM) \cite{FST1979} is
fundamental in studying various exactly solvable physical models
such as the XYZ model, six vertex model, eight vertex model, phase model, and lattice model etc. The most important aspect of the QISM is the algebraic Bethe Ansatz, which provides an effective
procedure to construct eigenvectors and calculate the eigenvalues for the Hamiltonian.

The main idea of the Bethe Ansatz in solving the phase model relies on two important operators
in the off-diagonals of the monodromy matrix. In the simplest situation, the monodromy matrix $T(x)$ is written as
\begin{equation}
T(x)=\begin{pmatrix} A(x) & B(x)\\ C(x) & D(x)\end{pmatrix}
\end{equation}
where the operators $B(x)$ and $C(x)$ act on the space of states $\mathcal V$ in the physical model. The $L$-matrix and therefore the monodromy matrices obey the famous RTT relation on $\mathcal V\otimes \mathcal V$:
\begin{align}\label{e:RLL}
R(x, y)L_1(x)L_2(y)=L_2(y)L_1(x)R(x, y),     %R(x, y)T_1(x)T_2(y)=T_2(y)T_1(x)R(x, y).
\end{align}
where the R-matrix satisfies the Yang-Baxter equation:
\begin{equation}
R(x, y)R(x, z)R(y, z)=R(y, z)R(x, z)R(x, y).
\end{equation}
%\underline{In a series of papers \cite{Whe2012, FWZ2009, FWZ2009b, FWZ2009c} Wheeler and collaborators}

Recently Bogoliubov \cite{Bog2005}, Foda, Wheeler and Zuparic \cite{Whe2012,FWZ2009,FWZ2009b, FWZ2009c}, Tsilevich\cite{Tsi2006} (see also \cite{BIK1997}) have established the relationship of the phase model and $i$-boson model with the KP and BKP hierarchies by the QISM and constructed
certain operators $\mathbb B(x),\mathbb C(x)$ and $\tilde{\mathbb B}(x),\tilde{\mathbb C}(x)$ using the
phase algebras and $i$-boson algebras. In the KP case with the R-matrix given in
\eqref{e:Rmat1}, they consider %the R-matrix
%{\color{red} where the $R$-matrix is given by
%\begin{align}
%R(x,y)=\begin{pmatrix}
%x&0&0&0\\
%0&0&x^{\frac12}y^{\frac12}&0\\
%0&x^{\frac12}y^{\frac12}&x-y&0\\
%0&0&0&x
%\end{pmatrix}.
%\end{align}
the $L$-matrix
$$\begin{matrix}L_m(x)=
 \begin{pmatrix}
x^{-\frac{1}{2}} & \psi^\dag_m \\
   \psi_m & x^{\frac{1}{2}}
\end{pmatrix}
\end{matrix}
$$
where $\psi_i, \psi_i^{\dag}$ obey the Heisenberg relation $[\psi_i,\psi^\dag_j]=\delta_{i,j}\pi_i$.
%together with $\mathcal{N}_i,\pi_i (0\leq i\leq M)$ generate the phase algebra with relations:
%%generators $\{\pi_i,\psi^\dag_i,\mathcal{N}_i,\psi_i\}_{0\leq i\leq M}$ of the phase algebras  %and $i$-boson algebras $\{\rho_i,\rho^\dag_i, \tilde{\mathcal{N}}_i\}_{0\leq i\leq M}$}
%%satisfy that
% \begin{gather*}
% [\psi_i,\psi^\dag_j]=\delta_{i,j}\pi_i,~[\mathcal{N}_i,\psi_j]=-\delta_{i,j}\psi_i,~[\mathcal{N}_i,\psi^\dag_j]=\delta_{i,j}\psi^\dag_i,~\psi_i\pi_i=\pi_i\psi^\dag_i=0, ~ 0\leq i,j\leq M.
% \end{gather*}
 Then the monodromy matrix
  $$ T(x)=L_M(x)\cdots L_1(x)L_0(x)=
 \begin{pmatrix}
   A(x) & B(x) \\
   C(x) & D(x)
  \end{pmatrix}$$
also satisfies the RTT equation \eqref{e:RLL} and
%{\color{red}\tilde{T}(x)=\tilde{L}_M(x)\cdots \tilde{L}_1(x)\tilde{L}_0(x)=
 %\begin{pmatrix}
%   \tilde{A}(x) & \tilde{B}(x) \\
 %  \tilde{C}(x) & \tilde{D}(x)
 % \end{pmatrix}.}
the operators
\begin{align}\label{in1}
\mathbb{B}(x)=x^{\frac{M}{2}}B(x), \text{and}~  \mathbb{C}(x)=x^{\frac{M}{2}}C(\frac{1}{x})
\end{align}
satisfy four distinguished properties \cite{Bog2005,FWZ2009,Whe2012} that can be used to characterize the KP system and compute the correlation functions. The exact meaning of the four properties will be presented in the sequel
(see \eqref{e:four}).
%{\color{red}\begin{enumerate}
% \item  $[\mathbb{B}(x),\mathbb{B}(y)]=[\mathbb{C}(x),\mathbb{C}(y)]=0$, i.e., $\mathbb{B}(x_1)\cdots \mathbb{B}(x_N)|0\rangle$ is a symmetric function in $x_1,x_2,\cdots x_N$;
%\item The limit of the operators are given by \begin{align}
%\lim_{M\rightarrow\infty}\mathbb{B}(x)=exp(\sum^{\infty}_{n=1}\frac{x^n}{n}H_{-n}),~~\lim_{M\rightarrow\infty}\mathbb{C}(x)=exp(\sum^{\infty}_{n=1}\frac{x^n}{n}H_n),
%\end{align}
%where $H_{-n},H_n$ satisfy the Heisenberg relation $[H_m, H_n]=m\delta_{m, -n}$;
%\item $\mathbb{B}(x_1)\cdots \mathbb{B}(x_N)|0\rangle$ is a KP $\tau$-function;
%\item $\mathbb{B}(x)|\nu)=\sum\limits_{\nu\prec \mu\subseteq [l+1,M]}x^{|\mu|-|\nu|}|\mu)$ .
%\end{enumerate}
%where their proof was based on the phase model and monodromy $L$-matrix\cite{Bog2005,Whe2012}.
% The two limits in (2) were called KP half-vertex operators \cite{FWZ2009,OR2003}. In \cite{OR2003}, Okounkov and Reshetikhin have shown that these half-vertex operators generate random plane partitions. Two operators $\mathbb{\tilde{B}}(x)=x^{\frac{M}{2}}\tilde{B}(x)$ and $\mathbb{\tilde{C}}(x)=x^{\frac{M}{2}}\tilde{C}(\frac{1}{x})$ were introduced in \cite{FWZ2007,Whe2012}.}

On one hand, the KP and its associate BKP can be formulated in the context of the boson-fermion correspondence (vertex operators realization
\cite{Jing1991,Jing1991a,KRR2013}) as special realizations of the basic representation of the affine Lie algebra \cite{Fr1981}, where the KP and BKP correspond the so-called bosonic/fermionic realization of the Lie algebra $\widehat{gl}_{\infty}$ and $\mathfrak{o}_\infty$\cite{DJKM6,
Kv1990, KRR2013}.

On the other hand, the boson-boson correspondence gives rise to interesting representations of the
$W_{1+\infty}$-algebra \cite{FKRW1995,Wang1998} including the Virasoro algebra
just as the boson-fermion correspondence. A bosonic KP hierarchy was proposed in \cite{Li2011}, and
in \cite{JL2019} we have shown there do not exist polynomial tau functions
in the bosonic system and computed some nontrivial tau functions in the completion of the Fock space of
the charged free bosons, which poses a question how to understand this new phenomenon.
The goal of this paper is to formulate the charged boson system using the idea of the QISM, specifically focusing on the operators that can be viewed as the off-diagonal entries of the monodromy matrix in the QISM and providing different perspective to understand the dynamic procedure.

In order to do this, we propose a QISM-like approach to the charged fermions by
generalizing the phase model in recognition that the atomic quadratic operators satisfy similar
bosonic relations. We define two new operators (see \eqref{B1}-\eqref{B4})
%The $bc$ fermionic fields (charged fermions) is closely linked to charged free bosons ($\beta\gamma$ systems) \cite{FKRW1995,Wang1998}. In the previous paper\cite{JL2019}, we studied bosonic KP tau functions in relation with the charged free bosonic fields. In this paper, we first introduce
\begin{align*}
\mathbf{B}(x)=\mathop{\overrightarrow\prod}\limits_{i\in (-\infty,M]}\exp\left(xb(-(i-1))c(i-2)\right), \qquad
\mathbf{C}(x)=\mathop{\overleftarrow\prod}\limits_{i\in (-\infty,M]}\exp\left(xb(-i+2)c(i-1)\right),
\end{align*}
where the fermionic operators obey the commutation relation
\begin{align}
\{b(i),c(j)\}=\delta_{i,-j},~~~~\{b(i),b(j)\}=\{c(i),c(j)\}=0,
\end{align}
in $bc$ fermionic fields for any positive integer $M$.
 We show that these two operators
 satisfy similar commutation relations for the monodromy matrix using the Baker-Campbell-Hausdorff (BCH) formula. Namely, we shall prove that the operators $\mathbf{B}(x), \mathbf{C}(x)$ coming from the $bc$ fermionic fields enjoy the following properties:
\begin{enumerate}
 \item  $\left[\mathbf{B}(x),\mathbf{B}(y)\right]=[\mathbf{C}(x),\mathbf{C}(y)]=0$, and
 $\mathbf{B}(x_1)\cdots \mathbf{B}(x_N)|0\rangle$ is a symmetric function in $x_1,\cdots x_N$;
\item The limit of the operators are given by \begin{align}\label{e:four}
\lim_{M\rightarrow\infty}\mathbf{B}(x)=\exp\left(\sum^{\infty}_{n=1}\frac{x^n}{n}H_{-n}\right),~~\lim_{M\rightarrow\infty}\mathbf{C}(x)=\exp\left(\sum^{\infty}_{n=1}\frac{x^n}{n}H_n\right),
\end{align}
where $H_{-n},H_n$ satisfy the Heisenberg relation $[H_m, H_n]=m\delta_{m, -n}$ (see \eqref{bc8});
\item The Bethe eigenvector $\mathbf{B}(x_1)\cdots \mathbf{B}(x_N)|0\rangle$ is a KP tau function\footnote{In the paper the tau function is an algebraic tau function without bosonization.};
\item $\mathbf{B}(x)|\nu)=\sum\limits_{\nu\prec \mu\subseteq [l+1,M]}x^{|\mu|-|\nu|}|\mu)$ (see \eqref{bc7}).
\end{enumerate}
Then the correlation function $\langle 0|\mathbf{C}(x_N)\cdots\mathbf{C}(x_1)\mathbf{B}(y_1)\cdots $ $\mathbf{B}(y_N)|0\rangle$
can be computed (cf. \cite{FWZ2009,Whe2012}).
%1. $[\mathbf{B}(x),\mathbf{B}(y)]=[\mathbf{C}(x),\mathbf{C}(y)]=0$, i.e., $\mathbf{B}(x_1)\cdots \mathbf{B}(x_N)|0\rangle$ is a symmetric function in $x_1,x_2,\cdots x_N$.\\
%2. $\mathbf{B}(x)|\nu)=\sum_{\nu\prec \mu\subseteq [l+1,M]}x^{|\mu|-|\nu|}|\mu)$ (see \ref{bc7}).\\
%3. Considering $M\rightarrow \infty$, we obtain
%\begin{align}
%\mathbf{B}(x)=exp(\sum^{\infty}_{n=1}\frac{x^n}{n}H_{-n}),~~\mathbf{C}(x)=exp(\sum^{\infty}_{n=1}\frac{x^n}{n}H_n),
%\end{align}
%where $H_{-n},H_n$ denote the Heisenberg generators (see \ref{bc8}).\\
%4. $\mathbf{B}(x_1)\cdots \mathbf{B}(x_N)|0\rangle$ is a KP $\tau$ function.
%\par In this paper, we first introduce two operators $\mathbf{B}(x)=\mathop{\overrightarrow\prod}\limits_{i\in (-\infty,M]}exp(xb(-(i-1))c(i-2))$ (see \ref{B1}) and $\mathbf{C}(x)=\mathop{\overleftarrow\prod}\limits_{i\in (-\infty,M]}exp(xb(-i+2)c(i-1))$ (see \ref{B4}) in $bc$ fermionic fields (we also denote them $\mathbf{B}(x)$ and $\mathbf{C}(x)$ in this paper). We find that these two operators also satisfy the four relations above.
%Then the correlation function (scalar product in \cite{FWZ2009,Whe2012}) $\langle 0|\mathbf{C}(x_N)\cdots\mathbf{C}(x_1)\mathbf{B}(y_1)\cdots \mathbf{B}(y_N)|0\rangle_{bc}$ follows. These two operators we introduce here may be equivalent to those in \cite{Bog2005,Whe2012}.

Our second main result is to formulate a QISM-like approach for the charged free bosonic system (or the bosonic $\beta\gamma$ system).
We will introduce two important operators (see \eqref{CB1}-\eqref{CB2})
\begin{align*}
\mathbf{B}^{*}(x)=\mathop{\overrightarrow\prod}\limits_{i\in (-\infty,M]}\exp\left(x\varphi_{-i}\varphi^*_{i-1}\right),\qquad \mathbf{C}^{*}(x)=\mathop{\overleftarrow\prod}\limits_{i\in (-\infty,M]}\exp\left(x\varphi_{-i+1}\varphi^*_{i}\right)
\end{align*}in terms of charged free bosons. Based on the BCH formula,  %and the properties of charged free bosons,
we will show that\\
1. $[\mathbf{B}^{*}(x),\mathbf{B}^{*}(y)]=[\mathbf{C}^{*}(x),\mathbf{C}^{*}(y)]=0$, so $\mathbf{B}^{*}(x_1)\cdots\mathbf{B}^{*}(x_N)|0\rangle$
is a symmetric function;\\
%, i.e., $\mathbf{B}^{*}(x_1)\cdots \mathbf{B}^{*}(x_N)|0\rangle$ is a symmetric function about $x_1,x_2,\cdots x_N$;\\
2. For $M\rightarrow \infty$,
\begin{align}
\mathbf{B}^{*}(x)=\exp\left(\sum^{\infty}_{n=1}\frac{(-1)^{n+1}}{n}x^nh_{-n}\right),~~\mathbf{C}^{*}(x)=\exp\left(\sum^{\infty}_{n=1}\frac{(-1)^{n+1}}{n}x^nh_n\right),
\end{align}
where the $h_n$ satisfies the Heisenberg relation $[h_m, h_n]=-m\delta_{m, -n}$ (see \eqref{eq3});\\
3. The Bethe vector $\mathbf{B}^{*}(x_1)\cdots \mathbf{B}^{*}(x_N)|0\rangle$ is a bosonic KP tau function and expressible in Schur functions as well;\\
4. For any positive integer $M$, the correlation function $\langle 0|\mathbf{C}^{*}(x_N)\cdots\mathbf{C}^{*}(x_1)\mathbf{B}^{*}(y_1)\cdots \mathbf{B}^{*}(y_N)|0\rangle$ is equal to the inverse of $\langle 0|\mathbf{C}(x_N)\cdots\mathbf{C}(x_1)\mathbf{B}(y_1)\cdots \mathbf{B}(y_N)|0\rangle$.
%\par And we give two operators
%\begin{align*}
%\mathbf{\tilde{B}}(x)=\mathop{\overrightarrow\prod}\limits_{j\in [-M+1,M]}exp(\frac{1}{2}x(-1)^j\phi_{-j+1}\phi_j)~ (\text{see}~ \ref{ne1}),~
%\mathbf{\tilde{C}}(x)=\mathop{\overleftarrow\prod}\limits_{i\in [-M+1,M]}exp(\frac{1}{2}x(-1)^{i-1}\phi_{-i}\phi_{i-1}) (\text{see}~ \ref{ne2})
%\end{align*}in neutral fermions. We also find that $\mathbf{\tilde{B}}(x)$ and $\mathbf{\tilde{C}}(x)$ have same features with $\mathbb{\tilde{B}}(x)$ and $\mathbb{\tilde{C}}(x)$ \cite{Whe2012} in some ways. Here $\mathbb{\tilde{B}}(x)=x^{\frac{M}{2}}\tilde{B}(x)$ and $\mathbb{\tilde{C}}(x)=x^{\frac{M}{2}}\tilde{C}(\frac{1}{x})$ are from the monodromy matrix
%$$ \tilde{T}(x)=\tilde{L}_M(x)\cdots \tilde{L}_1(x)\tilde{L}_0(x)=
% \left(
% \begin{matrix}
%   \tilde{A}(x) & \tilde{B}(x) \\
%   \tilde{C}(x) & \tilde{D}(x)
%  \end{matrix}
%  \right),
%$$
%where
% $$ \tilde{L}_m(x)=
% \left(
% \begin{matrix}
%   x^{-\frac{1}{2}} & 2^{\frac{1}{2}}\rho^\dag_m \\
%   2^{\frac{1}{2}}\rho_m & x^{\frac{1}{2}}
%  \end{matrix}
%  \right).
%$$

While fulfilling the above two goals, we also obtain and prove several fundamental combinatorial identities such
as the Cauchy identity as well as combinatorial formulae for the Schur symmetric polynomials and Schur's Q-functions, which might
offer new insight on these symmetric functions.

\par The paper is organized as follows. In section 2 (with Appendix B), we formulate an alternative QISM approach to the charged free fermions
by introducing two operators $\mathbf{B}(x)$ and $\mathbf{C}(x)$ %for the $bc$ fermionic fields
and deriving their important properties.
In particular, we give two commutative diagrams with
the operators $\mathbb{B}(x)$ and $\mathbb{C}(x)$ considered in \cite{Bog2005,FWZ2009,Whe2012}, which then gives a new method to
compute the correlation functions.
In section 3, we introduce the bosonic analog $\mathbf{B}^{*}(x)$ and $\mathbf{C}^{*}(x)$ in terms of
charged free bosons, and use them to compute correlation functions
$\langle 0|\mathbf{C}^{*}(x_N)\cdots\mathbf{C}^{*}(x_1)\mathbf{B}^{*}(y_1)\cdots \mathbf{B}^{*}(y_N)|0\rangle$
and relate them to their fermionic counterparts. In section 4 (with Appendix C), we define two operators $\mathbf{\tilde{B}}(x)$ and $\mathbf{\tilde{C}}(x)$ from neutral fermions and establish the connection with $\mathbb{\tilde{B}}(x)$ and $\mathbb{\tilde{C}}(x)$
as well as computing the correlation functions.

\section{$\mathbf{B}(x)$ and $\mathbf{C}(x)$ of $bc$ fermionic fields and correlation functions}\label{se1}
In this section we introduce operators $\mathbf{B}(x)$ and $\mathbf{C}(x)$ of the $bc$ fermionic fields and obtain their relations based on the Baker-Campbell-Hausdorff (BCH) formula. They correspond to the off-diagonal operators of the monodromy matrix in the phase model (cf. Appendix B).
%We also establish the correspondence between $\mathbf{B}(x), \mathbf{C}(x)$ and $\mathbb{B}(x)$, $\mathbb{C}(x)$.
%We discuss the actions of $\mathbf{B}(x)$ on $\mathcal{F}^{(0)}$ and $\mathbf{C}(x)$ on $\mathcal{F}^{*(0)}$ in Section \ref{sub2}. In Section \ref{sub3} we relate $\mathbf{B}(x)$ and $\mathbf{C}(x)$ with $\mathbb{B}(x)$ and $\mathbb{C}(x)$. In Section \ref{sub4} we give two expressions of correlation function $\langle 0|\mathbf{C}(x_N)\cdots\mathbf{C}(x_1)\mathbf{B}(y_1)\cdots \mathbf{B}(y_N)|0\rangle_{bc}$, which is a key tool in Section 3.
\subsection{Operators $\mathbf{B}(x)$ and $\mathbf{C}(x)$}\label{sub1}
Let the $bc$ fermionic fields be
\begin{align}
b(z)=\sum_{i\in \mathbb{Z}}b(i)z^{-i},~~~~c(z)=\sum_{i\in \mathbb{Z}}c(i)z^{-i-1}
\end{align}
with the commutation relations
\begin{align}\label{bc1}
\{b(i),c(j)\}=\delta_{i,-j},~~~~\{b(i),b(j)\}=\{c(i),c(j)\}=0,
\end{align}
where $\{A,B\}=AB+BA$. In particular, $b^2(i)=c^2(i)=0$ for all $i\in \mathbb{Z}$.
% Their nontrivial communications can be written in the operator product expansion (OPE) as follows.
%\begin{align*}
%b(z)c(w)\sim \frac{1}{z-w}, ~c(z)b(w)\sim \frac{1}{z-w}.
%\end{align*}
\par For any fixed $M\in\mathbb N$ we define two operators $\mathbf{B}(x)$ and $\mathbf{C}(x)$ in terms of the $bc$ fermionic fields:
\begin{align}\notag
\label{B1}\mathbf{B}(x)&=\exp\left(xb(-M+1)c(M-2)\right)\exp\left(xb(-M+2)c(M-3)\right)\cdots\\
&=\mathop{\overrightarrow\prod}\limits_{i\in (-\infty,M]}\exp\left(xb(-(i-1))c(i-2)\right),\\
\label{B4}\mathbf{C}(x)&=\cdots \exp\left(xb(-M+3)c(M-2)\right)\exp\left(xb(-M+2)c(M-1)\right) \notag \\
&=\mathop{\overleftarrow\prod}\limits_{i\in (-\infty,M]}\exp\left(xb(-i+2)c(i-1)\right),
\end{align}
where the product $\mathop{\overrightarrow\prod}\limits_{i\in [a,M]}$ (resp.$\mathop{\overleftarrow\prod}\limits_{i\in [a,M]}$) runs from left to right (resp. right to left) as $i$ goes down from $M$ to $a$. Here we always use $a<M$ and if $a=-\infty, [a, M]=(-\infty, M]$.
%decreasingly %in the decreasing %(resp. increasing) order
%within the interval $(-\infty, M]$.
%from left to right within the interval $(-\infty, M]$.
\begin{remark} Though the operator $\mathbf{B}(x)$ (resp. $\mathbf{C}(x)$) involves an infinite sum,  its action on our concerned space $\mathcal{F}^{(0)}$ (resp. $\mathcal{F}^{*(0)}$) is finite (see Sect. \ref{sub2}).
%\uline{In fact, it agrees with the action of $\mathbb{B}(x)$ on $\mathcal{F}^{(0)}$ (resp. $\mathbb{C}(x)$ on $\mathcal{F}^{*(0)}$). }{\color{red}delete}
\end{remark}
\begin{proposition}\label{pro1} For any fixed $M$, one has that
\begin{align}\label{bc9}
\mathbf{B}(x)=\exp\left(\sum^{\infty}_{n=1}\frac{1}{n}x^n\wedge_n\right),~~\mathbf{C}(x)=\exp\left(\sum^{\infty}_{n=1}\frac{1}{n}x^n\wedge^*_n\right),
\end{align}
where $\wedge_n=\sum^{M-1}_{-\infty}b(-i)c(i-n),~\wedge^*_n=\sum^{M-1}_{-\infty}b(-i+n)c(i).$
\end{proposition}
\begin{proof} We divide the proof into two steps.
%First we recall the formula to derive its consequence Lemma \ref{le2}. Then we prove (\ref{bc9}) using the BCH formula and Lemma \ref{le2}.

\textbf{Step 1.}
Let $X$ and $Y$ be operators on a Hilbert space. The \textit{Baker-Campbell-Hausdorff (BCH) formula}\cite[pp.162-173]{Jac1962}\cite{Tho1982} says that $\exp(X)\exp(Y)= \exp(Z)$, where $Z$ is a formal series in iterated commutators of $X$ and $Y$ with rational coefficients.
%It is known that $Z$ can be written as a formal series in  of $X$ and $Y$.
The first few terms are
\begin{align*}
Z=X+Y+\frac{1}{2}[X,Y]+\frac{1}{12}\left[X,[X,Y]\right]+\frac{1}{12}\left[[X,Y],Y\right]+\cdots
\end{align*}

Write $[X^{(n)},Y]$ as $[X,[X^{(n-1)},Y]]$ inductively, thus
$[X^{(3)},Y]=\left[X,[X,[X,Y]]\right]$.
Similarly, we denote $[X,Y^{(n)}]=[[X,Y^{(n-1)}],Y],~[X,Y^{(n)}X]=[[X,Y^{(n)}],X]$.
Let $C_n$ (resp. $D_n$) be the coefficients of $[X^{(n)},Y]$ (resp.$[X,Y^{(n)}]$) in $Z$, it is known that
\begin{align}\label{BCH1}
C_n=D_n=\frac{(-1)^{n}}{n!}B_{n},
\end{align}
where $B_{n}$ are the Bernoulli numbers\cite{Le1935} defined by
\begin{align}\label{BCH2}
\frac{z}{e^{z}-1}=\sum^{\infty}_{n=0}B_{n}\frac{z^n}{n!},~~~~|z|<2\pi.
\end{align}

Let $f(x)=\sum_{n\geq 1}\frac{1}{n}x^{n}=-\text{ln}(1-x)$ and consider the Taylor expansion of $(f(x))^{s}~(s\geq 1)$
%\begin{align}
%f(x)=\sum_{l\geq 1}\frac{1}{l}x^{l}=-\text{ln}(1-x),~~|x|<1
%\end{align}
%and we define $F^{s}_{t}$ by the following generating function
\begin{align}\label{eq4}
\left(f(x)\right)^{s}=\sum_{t\in \mathbb{Z}_+}F^{s}_{t}x^{t},
\end{align}
then $F^{s}_{t}=0$ for $t<s$ and $F^{s}_{s}=1$.
We have the following technical lemma.
\begin{lemma}\label{le2}
The coefficients $C_{j}$ and $F^j_n$ satisfy the following identities:
\begin{align}\label{BCH5}
\sum^{n}_{j=1}C_{j}F^{j}_{n}=\frac{1}{n+1},\qquad\qquad \sum^{t}_{s=1}\frac{F^{s}_{t}}{s!}=1.
\end{align}
\end{lemma}
\begin{proof}
Setting $z=f(x)=-\text{ln}(1-x)$, we have
\begin{align}\label{BCH3}
\sum_{j\geq 1}C_{j}z^{j}=\frac{z}{1-e^{-z}}-1=-\frac{\text{ln}(1-x)}{x}-1=\sum_{n\geq 1}\frac{x^n}{n+1}
\end{align}
by combining (\ref{BCH1}) and (\ref{BCH2}).
And by (\ref{eq4}),
\begin{align}\label{BCH4}
\sum_{j\geq 1}C_{j}z^{j}=\sum_{j\geq 1}C_j\sum_{n\geq j}F^j_nx^n=\sum_{n\geq 1}x^n\sum^n_{j=1}C_jF^j_n.
\end{align}

Comparing (\ref{BCH3}) and (\ref{BCH4}), we get the first identity. %(\ref{BCH5}).
The second one can be checked similarly. %Similarly we can check (\ref{BCH6}).
\end{proof}
%\subsection{Operators $\mathbf{B}(x)$ and $\mathbf{C}(x)$}\label{sub7}
\textbf{Step 2.} We claim that for $m\leq M-1$,
\begin{align}\label{bc5}
\mathop{\overrightarrow\prod}\limits_{i\in [m,M-1]}\exp\left(xb(-i)c(i-1)\right)
%=\exp\left(xb(-M+1)c(M-2)\right)\cdots \exp\left(xb(-m)c(m-1)\right)\\
=\exp\left(\sum^{M-m}_{n=1}\frac{x^{n}}{n}\sum^{M-1}_{j=n+m-1}b(-j)c(j-n)\right).
\end{align}

We use induction on $m$. First, for $m=M-2$, set $X=xb(-M+1)c(M-2),~Y=xb(-M+2)c(M-3)$. As $[X, Y]$ commutes with $X$ and $Y$, so
%then the only nonzero iterated commutators of $X$ and $Y$ in $Z$ is $[X,Y]$.
%\begin{align*}
%[X,Y]=x^2b(-M+1)c(M-3).
%\end{align*}
%So we have
\begin{align*}
Z=&X+Y+\frac{1}{2}[X,Y]
%=&xb(-M+1)c(M-2)+xb(-M+2)c(M-3)+\frac{1}{2}x^2b(-M+1)c(M-3)\\
=\sum^{2}_{n=1}\frac{x^{n}}{n}\sum^{M-1}_{j=n+M-2-1}b(-j)c(j-n).
\end{align*}

Assume (\ref{bc5}) holds for $m=k+1$, set $X=\sum^{M-k-1}_{n=1}\frac{x^{n}}{n}\sum^{M-1}_{j=n+k}b(-j)c(j-n),~Y=xb(-k)c(k-1)$, then the only nontrivial iterated commutators of $X$ and $Y$ in $Z$ are
\begin{align}
[X^{(s)},Y]=\sum^{M-1-k-s}_{j=0}F^s_{s+j}x^{s+j+1}b(-k-s-j)c(k-1),~~~~~1\leq s\leq M-1-k,
\end{align}
so
\begin{align*}
Z&   %=X+Y+\sum^{M-1-k}_{s=1}C_s[X^{(s)},Y]\\
=X+Y+\sum^{M-1-k}_{s=1}C_s\sum^{M-1-k-s}_{j=0}F^s_{s+j}x^{s+j+1}b(-k-s-j)c(k-1)\\
&=X+Y+\sum^{M-1-k}_{t=1}\sum^t_{s=1}C_sF^s_tx^{t+1}b(-k-t)c(k-1)\\
&=X+Y+\sum^{M-1-k}_{t=1}\frac{x^{t+1}}{t+1}b(-k-t)c(k-1)\\
%&=X+\sum^{M-k}_{n=1}\frac{x^{n}}{n}b(-n-k+1)c(n+k-1-n)\\
&=\sum^{M-k}_{n=1}\frac{x^{n}}{n}\sum^{M-1}_{j=n+k-1}b(-j)c(j-n),
\end{align*}
where we have used Lemma \ref{le2} (also recalled X and Y). Therefore \eqref{bc5} is true.
%Thus we have $$\mathop{\overrightarrow\prod}\limits_{i\in [k,M-1]}\exp\left(xb(-i)c(i-1)\right)=\exp\left(\sum^{M-k}_{n=1}\frac{x^{n}}{n}\sum^{M-1}_{j=n+k-1}b(-j)c(j-n)\right).$$
The proposition follows by taking $m\rightarrow -\infty$ in \eqref{bc5}.
\end{proof}
\begin{proposition} For fixed $M$, one has that
\begin{align}
\mathbf{B}(x)\mathbf{B}(y)=\mathbf{B}(y)\mathbf{B}(x),~~~~\mathbf{C}(x)\mathbf{C}(y)=\mathbf{C}(y)\mathbf{C}(x).
\end{align}
\end{proposition}
\begin{proof}
Since
\begin{align}\label{eq5}
[\wedge_m,\wedge_n]=\sum^{M-1}_{-\infty}b(-i)c(i-m-n)-\sum^{M-1}_{-\infty}b(-i)c(i-m-n)=\wedge_{m+n}-\wedge_{m+n}=0,
\end{align}
we have $[\sum^{\infty}_{n=1}\frac{1}{n}x^n\wedge_n,\sum^{\infty}_{n=1}\frac{1}{n}y^n\wedge_n]=0$, then
the commutation relations hold.
%\begin{align*}
%\mathbf{B}(x)\mathbf{B}(y)=\mathbf{B}(y)\mathbf{B}(x).
%\end{align*}
\end{proof}
\begin{proposition}
Setting $\lim\limits_{M\rightarrow \infty}\wedge_n=H_{-n},~\lim\limits_{M\rightarrow \infty}\wedge^*_n=H_n$, then we have that
\begin{align}\label{bc8}
[H_m,H_n]=m\delta_{m,-n}.
\end{align}
\end{proposition}
Thus we have the following identities.
\begin{proposition}\label{pro2} The limit of the operators satisfy that
\begin{align}
&\label{bc10}\lim\limits_{M\rightarrow \infty}\mathbf{B}(x)=\exp\left(\sum^{\infty}_{n=1}\frac{x^n}{n}H_{-n}\right),~~\lim\limits_{M\rightarrow \infty}\mathbf{C}(x)=\exp\left(\sum^{\infty}_{n=1}\frac{x^n}{n}H_n\right),\\
&\label{bc6}\lim\limits_{M\rightarrow \infty}\mathbf{C}(x)\mathbf{B}(y)=\frac{1}{1-xy}\lim\limits_{M\rightarrow \infty}\mathbf{B}(y)\mathbf{C}(x).
\end{align}
\end{proposition}
\subsection{Action of $\mathbf{B}(x)$ on $\mathcal{F}^{(0)}$ and $\mathbf{C}(x)$ on $\mathcal{F}^{*(0)}$}\label{sub2}
The Fock space $\mathcal{F}$ is spanned by negative (resp. non-positive) modes of $c(z)$ (resp. $b(z)$) acting on the vacuum vector $|0\rangle$ satisfying the relations %satisfying
\begin{align}\label{bc2}
b(m+1)|0\rangle=0,~~~~c(m)|0\rangle=0,~~~~m\geq 0.
\end{align}
Then $\mathcal F$ is spanned by $b(-m_1)\cdots b(-m_s)c(-n_1)\cdots c(-n_t)|0\rangle$,
% $i_1>\cdots >i_s\geq 0$ and $j_1>\cdots >j_t>0$ (arbitrary $s,t\geq 1$)
and decomposes as follows. %itself into a direct sum
\begin{align*}
\mathcal{F}=\bigoplus_{i\in \mathbb{Z}}\mathcal{F}^{(i)},
\end{align*}
where the subspace $\mathcal{F}^{(i)}$ has a basis consisting of monomials $ b(-m_1)\cdots b(-m_s)c(-n_1)\cdots c(-n_t)|0\rangle$,
$m_1>\cdots >m_s\geq 0$ and $n_1>\cdots >n_t>0$ such that $s-t=i$. Let's focus on the subspace $\mathcal{F}^{(0)}$.
%of $\mathcal{F}$ with a basis of the elements
%\begin{align*}
%b(-i_1)\cdots b(-i_s)c(-j_1)\cdots c(-j_s)|0\rangle
%\end{align*}
%where $i_1>i_2>\cdots >i_s\geq 0,~j_1>j_2>\cdots j_s>0$ including the vacuum vector.
\par The dual Fock space $\mathcal{F}^{*}$ is generated by the right action of $b(n), c(n)$ on
 the dual vacuum vector $\langle 0|$ satisfying the relations %satisfying
\begin{align}
\langle 0|b(i+1)=0,~~~~\langle 0|c(i)=0,~~~~i< 0.
\end{align}

We are only interested in the subspace $\mathcal{F}^{*(0)}$ with a basis
$\langle 0|b(m_k)\cdots b(m_1)c(n_k)\cdots c(n_1)$, \newline
$m_1>\cdots >m_k> 0,~n_1>\cdots >n_k\geq0$.

\par A vector $\tau\in \mathcal{F}$ is called a KP tau function\cite{KRR2013,Whe2012} if $\tau$ obeys
\begin{align}
\mathrm{Res}_{z}b(z)\otimes c(z)(\tau\otimes \tau)=0,
\end{align}
where $\mathrm{Res}_{z}f(z)$ denotes the coefficient of $z^{-1}$ in $f(z)$.
%\begin{remark}

By the vacuum condition, $\mathrm{Res}_{z}b(z)\otimes c(z)\left(|0\rangle\otimes |0\rangle\right)=0$. Further, it can be shown that
\cite{JL2019}
\begin{align*}
\left[\mathrm{Res}_{z}b(z)\otimes c(z),\mathbf{B}(x)\otimes \mathbf{B}(x)\right]=0,
\end{align*}
thus $\mathbf{B}(x_1)\mathbf{B}(x_2)\cdots \mathbf{B}(x_N)|0\rangle$ is a KP tau function, and can be called a Bethe eigenvector
(cf. \cite{Whe2012}). %Prop. \ref{pro3},}
% $\mathbf{B}(x_1)\mathbf{B}(x_2)\cdots \mathbf{B}(x_N)|0\rangle$ .

A \textit{partition} $\mu=(\mu_1\ldots\mu_l)$   %\{\mu_1\geq\cdots \geq\mu_l>\mu_{l+1}=\cdots=0\}$
is a weakly decreasing non-negative integers. Its weight
is $|\mu|=\sum^{l}_{j=1}\mu_{j}$ and the length is $l(\mu)=l$.
Pick a rectangle $[N,M]$ containing the Young diagram of $\mu$, i.e., $\mu_1\leq M,~l(\mu)\leq N$.
For any partition $\mu=(\mu_1\ldots\mu_l)$, we define
\begin{align}
\label{bc14}|\mu)&=|\mu_1,\dots,\mu_l)=b(-m_1)b(-m_2)\cdots b(-m_l)c(-l)c(-l+1)\cdots c(-1)|0\rangle,\\
\label{bc15}(\mu|&=(\mu_1,\dots,\mu_l|=\langle 0|b(1)b(2)\cdots b(l)c(m_l)c(m_{l-1})\cdots c(m_1),
\end{align}
where $m_i=\mu_i-i$ for all $1\leq i\leq l$, so $m_{1}>\cdots >m_{l}>-l$.
We also define $deg\left(|\mu)\right)=deg\left((\mu|\right)=|\mu|$.
 %and $|\mu)\in \mathcal{F}^{(0)},~ (\mu| \in \mathcal{F}^{*(0)}$. %Then we have $|\mu)\in \bigwedge^{\infty}_{(0)}V$ and
%We also define $deg(|\mu))=deg((\mu|)=|\mu|$.
\begin{definition}
The partition $\mu=(\mu_{1},\dots,\mu_{l+1})$ is said to interlace the partition $\nu=(\nu_{1},\dots,\nu_{l})$, written as
$\mu\succ \nu$, if for all $1\leq i\leq l$ %and only if
\begin{align}
\mu_{i}\geq\nu_{i}\geq\mu_{i+1}.
\end{align}\end{definition}

\begin{theorem}\label{th4}
For fixed $M\in\mathbb N$ and vector $|\nu)=|\nu_{1},\dots,\nu_{l})$ such that $M\geq\nu_1$ %\geq\nu_l$
we have that
\begin{align}\label{bc7}
\mathbf{B}(x)|\nu)=\sum_{\nu\prec \mu\subseteq [l+1,M]}x^{|\mu|-|\nu|}|\mu).
\end{align}
\end{theorem}
\begin{proof} It follows from (\ref{bc1}) that %and (\ref{bc2}) that
\begin{align}
\exp\left(xb(-j)c(j-1)\right)=1+xb(-j)c(j-1).
\end{align}

As $b(-k)c(k-1)|\nu)=0$ for $k\leq -l-1$, we have that
\begin{align*}
\mathbf{B}(x)|\nu)&=\mathop{\overrightarrow\prod}\limits_{j\in [-l,M-1]}\exp\left(xb(-j)c(j-1)\right)|\nu)=\mathop{\overrightarrow\prod}\limits_{j\in [-l,M-1]}\left(1+xb(-j)c(j-1)\right)|\nu)\\
&=\sum_{1\leq s\leq M+l} x^{s}b(-a_1)c(a_1-1)b(-a_2)c(a_2-1)\cdots b(-a_s)c(a_s-1)|\nu)+|\nu),
\end{align*}
where $a_{M+l}< a_{M+l-1}<\cdots< a_2< a_1\leq M-1$ and $b(-a_i)c(a_i-1)$ are bundled together to act.\\
%$a_1\leq M-1,~a_{t+1}\leq a_t-1,~s\leq M+l$. The action order of $b(-a_1)c(a_1-1)b(-a_2)c(a_2-1)\cdots b(-a_s)c(a_s-1)|\nu)$ is $b(-a_s)c(a_s-1)$ act on $|\nu)$ first, then $b(-a_{s-1})c(a_{s-1}-1)$ act on $b(-a_s)c(a_s-1)|\nu)$, until $b(-a_1)c(a_1-1)$ act on $b(-a_2)c(a_2-1)\cdots b(-a_s)c(a_s-1)|\nu)$.
For simplicity, we denote $|-l\rangle=c(-l)c(-l+1)\cdots c(-1)|0\rangle$ in the following.
\par On one hand, for
$ |\lambda)=b(-s_1)\cdots b(-s_l)|-l\rangle,~~\lambda_{i}=s_i+i ~\text{and}~s_1>\cdots>s_{l}\geq -l+1$,
we have
$$\begin{cases}
b(-s_i-1)c(s_i)(|\lambda))=|\lambda_1,\cdots,\lambda_{i-1},\lambda_{i}+1,\lambda_{i+1},\cdots, \lambda_{l}), \lambda_{i-1}-\lambda_i> 0,\\
b(l)c(-l-1)(|\lambda))=|\lambda_1,\cdots,\lambda_{l},1),\\
deg(b(-s_i-1)c(s_i)(|\lambda)))=deg(b(l)c(-l-1)(|\lambda)))=deg(|\lambda))+1.
\end{cases}
$$
%\begin{align*}
%b(-j)c(j-1)(|\lambda))=&\sum^l_{i=1}b(-s_1)\cdots b(-s_{i-1})[b(-j)c(j-1),b(-s_i)]b(-s_{i+1})\cdots b(-s_l)|-l\rangle\\
%+&b(-s_1)b(-s_2)\cdots b(-s_l)b(-j)c(j-1)|-l\rangle.
%\end{align*}
%Since
%$$
%[b(-j)c(j-1),b(-s_i)]=\begin{cases}
%0& j\neq s_i+1\\
%b(-s_i-1)& j=s_i+1,
%\end{cases}
%$$
%we can get
%$$
%b(-j)c(j-1)(|\lambda))=\begin{cases}
%b(-s_1)\cdots b(-s_{i-1})b(-s_i-1)b(-s_{i+1})\cdots b(-s_l)|-l\rangle & j=s_i+1 ~and~ s_{i-1}-s_{i}>1\\
%b(-s_1)b(-s_2)\cdots b(-s_l)b(-l)|-l-1\rangle & j=-l\\
%0 & otherwise
%\end{cases}
%$$
%i.e.,
due to the fact that $[b(-j)c(j-1),b(-i)]=\delta_{i+1,j}b(-j)$ and $c(j)|0\rangle=0,~j\geq 0$.
Note that if $l_{i-1}=s_{i-1}-s_{i}>1$, we have
\begin{align*}
&\mathop{\overrightarrow\prod}\limits_{k\in [0,l_{i-1}-2]}b(-s_i-k-1)c(s_i+k)(|\lambda))=|\lambda_1,\cdots,\lambda_{i-1},\lambda_{i}+l_{i-1}-1,\lambda_{i+1},\cdots, \lambda_{l}),\\
&\mathop{\overrightarrow\prod}\limits_{k\in [0,l_{i-1}-1]}b(-s_i-k-1)c(s_i+k)(|\lambda))=0.
\end{align*}

These relations imply that the new element $|\mu)$ generated by $b(-j)c(j-1)$ satisfies that
\begin{align}
\mu_{i}\geq\nu_{i}\geq\mu_{i+1},~~1\leq i\leq l.
\end{align}
In other words, for $|\mu)=b(-a_1)c(a_1-1)b(-a_2)c(a_2-1)\cdots b(-a_s)c(a_s-1)|\nu)$, we have $\nu\prec \mu$ and $|\mu|-|\nu|=s$.

On the other hand, given an element $|\mu)$ with $\nu\prec \mu$, say
\begin{align}
&|\mu)=b(-m_1)b(-m_2)\cdots b(-m_l)b(-m_{l+1})c(-l-1)c(-l)c(-l+1)\cdots c(-1)|0\rangle,\\
&|\nu)=b(-n_1)b(-n_2)\cdots b(-n_l)b(-n_{l+1})c(-l-1)c(-l)c(-l+1)\cdots c(-1)|0\rangle,
\end{align}
where $m_{i}=\mu_i-i,~ n_{i}=\nu_i-i,~ -l+1\leq n_{l},~n_{l+1}=l+1\geq -m_{l+1}$, and $m_{i}\geq n_{i}> m_{i+1}$.

For $1\leq i\leq l$, we set
$$E_{i}=\begin{cases}
1,& m_i=n_i\\
b(-m_i)c(m_i-1)b(-m_i+1)c(m_i-2)\cdots b(-n_i-1)c(n_i), & m_i>n_i,
\end{cases}$$
then
\begin{align*}
E_{1}E_{2}\cdots E_{l+1}|\nu)=|\mu).
\end{align*}
Note that $E_{1}E_{2}\cdots E_{l+1}$ is generated by $b(-j)c(j-1),~ -l\leq j\leq M-1$. This completes the proof.
\end{proof}
Using the same method, we have the following result.
\begin{corollary}\label{cor2}
For arbitrary positive integer $M$ and vector $(\nu|=(\nu_1,\dots,\nu_l|$
%\begin{align*}
%(\nu|=(\nu_1,\dots,\nu_l|=\langle 0|b(1)\cdots b(l-1)b(l)c(\nu_l-1)\cdots c(\nu_2-2)c(\nu_1-1),~~0<\nu_l\leq\nu_1\leq M,
%\end{align*}
we have
\begin{align}
(\nu|\mathbf{C}(x)=\sum_{\nu\prec \mu\subseteq [l+1,M]}x^{|\mu|-|\nu|}(\mu|.
\end{align}
\end{corollary}

Following \cite{Mac1995}, the \textit{complete symmetric function} $h_{k}(x)$ in the variables $x_1,x_2,\cdots$ is given by
\begin{align*}
\sum_{k=0}^{\infty}h_{k}(x)z^{k}=\prod^{\infty}_{i=1}\frac{1}{1-x_iz}.
\end{align*}

To each partition $\lambda=(\lambda_1,\lambda_2,\dots,\lambda_k)$ we associate the \textit{Schur function} $s_\lambda(x)$ defined by
\begin{align*}
s_{\lambda}(x)=\det\left(h_{\lambda_{i}-i+j}(x)\right)_{1\leq i,j\leq k}.
\end{align*}

For the rest of the paper, we usually consider the Schur function $s_\lambda\{x\}$ in finitely many variables $\{x\}=\{x_1,x_2,\cdots ,x_N\}$,
which is obtained by letting $x_{N+1}=x_{N+2}=\cdots=0$ in $s_{\lambda}(x)$.
The following identities are well-known \cite{Mac1995}:
\begin{align}\label{bc3}
&s_\mu\{x\}=\sum_{\nu\prec \mu}s_\nu\{\bar{x}\}x^{|\mu|-|\nu|}_N, \\
\label{bc11}&s_\mu\{x\}=0, \qquad l(\mu)>N.
\end{align}
where  $\{\bar{x}\}=\{x\}\backslash \{x_N\}=\{x_1,\cdots, x_{N-1}\},  $.
\par Using Proposition \ref{pro1} and \cite{Whe2012}, we have
\begin{proposition}\label{e:genSchur}
For a fixed positive integer $M$ and $\{x\}=\{x_1,\cdots, x_N\}$,
\begin{align}
\label{bc12}\mathbf{B}(x_1)\cdots \mathbf{B}(x_N)|0\rangle=&\exp\left(\sum_{\substack{1\leq m\leq M \\ 1\leq n\leq N}}(-1)^{n-1}s_{(m,1^{n-1})}\{x\}b(-m+1)c(-n)\right)|0\rangle\\
\notag=&\sum_{\mu\subseteq[N,M]}s_\mu\{x\}|\mu),\\
\label{bc13}\langle 0|\mathbf{C}(x_N)\cdots\mathbf{C}(x_1)=&\langle 0|\exp\left(\sum_{\substack{1\leq m\leq M \\ 1\leq n\leq N}}(-1)^{n-1}s_{(m,1^{n-1})}\{x\}b(n)c(m-1)\right)\\
\notag=&\sum_{\mu\subseteq[N,M]}s_\mu\{x\}(\mu|.
\end{align}

% $s_{\{m,1^{n-1})}(x)$ is the schur function associated to partition $\{m,\underbrace{1\cdots 1}_{n-1}\}$ in the variables $x_1,x_2,\cdots x_{N}$.
\end{proposition}
\subsection{Correlation function $\langle 0|\mathbf{C}(x_N)\cdots\mathbf{C}(x_1)\mathbf{B}(y_1)\cdots \mathbf{B}(y_N)|0\rangle$}\label{sub4}
For $(\lambda|\in \mathcal{F}^{(0)}$ and $|\mu)\in \mathcal{F}^{*(0)}$, we define the bilinear pairing $\langle\ |\ \rangle$ by
\begin{align}
\label{bc16} %(\lambda,\mu)=
(\lambda|\mu)=\delta_{\lambda,\mu},
\end{align}
where it is assumed that $\langle0|1|0\rangle=1$.
\begin{theorem}\label{th3}
For a fixed positive integer $M$,
\begin{align}
\label{eq1}&\langle 0|\mathbf{C}(x_N)\cdots\mathbf{C}(x_1)\mathbf{B}(y_1)\cdots \mathbf{B}(y_N)|0\rangle =\sum_{\mu\subseteq[N,M]}s_\mu\{x\}s_\mu\{y\},
\end{align}
where $\{x\}=\{x_1,\cdots, x_N\}$, $\{y\}=\{y_1,\cdots, y_N\}$, and $\mu$ runs through all Young diagrams inside the rectangle
$[N, M]$ of height $N$ and width $M$. In particular, when $M\rightarrow\infty$, we get the Cauchy identity
\begin{align}
\label{eq1b}&\langle 0|\mathbf{C}(x_N)\cdots\mathbf{C}(x_1)\mathbf{B}(y_1)\cdots \mathbf{B}(y_N)|0\rangle =\prod_{i, j=1}^N\frac1{1-x_iy_j}=\sum_{l(\mu)\leq N}s_\mu\{x\}s_\mu\{y\}.
\end{align}
\end{theorem}
\begin{remark}\label{RE3}
%Hypergeometric functions are important in physics and mathematics.
The correlation function \eqref{eq1} is also known as an example of hypergeometric tau function %(3.1.12)
in \cite{OS2001} (see also \cite{HO2015}). In fact, \eqref{eq5} implies that %we have
%\begin{align*}
%\mathbf{C}(x_N)\cdots\mathbf{C}(x_1)=\exp\left(\sum^{\infty}_{n=1}\frac{\sum^n_{i=1}x^n_i}{n}\wedge^*_n\right),~~
%\mathbf{B}(y_1)\cdots \mathbf{B}(y_N)=\exp\left(\sum^{\infty}_{n=1}\frac{\sum^n_{i=1}x^n_i}{n}\wedge_n\right),
%\end{align*}
%thus
\begin{align*}
\langle 0|\mathbf{C}(x_N)\cdots\mathbf{C}(x_1)\mathbf{B}(y_1)\cdots \mathbf{B}(y_N)|0\rangle=\langle 0|\exp\left(\sum^{\infty}_{n=1}\frac{\sum^N_{i=1}x^n_i}{n}\wedge^*_n\right)\exp\left(\sum^{\infty}_{n=1}\frac{\sum^N_{i=1}y^n_i}{n}\wedge_n\right)|0\rangle,
\end{align*}
which agrees with \cite[Eq.(3.1.12)]{OS2001} by setting $r(x)$ to unity in the range of arguments from $-\infty$ to $M$ and zero elsewhere.
Then one can get the function in \cite[Eq. (3.1.4)]{OS2001} (cf.\cite{HO2015,NO2017}), which can be used to give another proof of \eqref{eq1}.
\end{remark}
%Taking the limit $M\rightarrow \infty,N\rightarrow \infty$ we can get from (\ref{bc6})
\begin{remark}\label{RE1} Using \eqref{bc6} we get the following
\begin{align}
\lim_{M, N\rightarrow \infty}\langle 0|\mathbf{C}(x_N)\cdots\mathbf{C}(x_1)\mathbf{B}(y_1)\cdots \mathbf{B}(y_N)|0\rangle=\prod^{\infty}_{i,j=1}\frac{1}{1-x_iy_j}=\sum_{\mu}s_{\mu}(x)s_{\mu}(y).
\end{align}
In particular, for $x_i=y_i=z^{i-\frac{1}{2}}$, we have
\begin{align*}
\lim_{M\rightarrow \infty,N\rightarrow \infty}\langle 0|\mathbf{C}(z^{N-\frac{1}{2}})\cdots\mathbf{C}(z^{1-\frac{1}{2}})\mathbf{B}(z^{1-\frac{1}{2}})\cdots \mathbf{B}(z^{N-\frac{1}{2}})|0\rangle=\prod^{\infty}_{i=1}
\frac{1}{(1-z^i)^i},
\end{align*}
which is the generating function of plane partitions \cite{Mac1995} (cf. \cite{Whe2012}).
\end{remark}
%\begin{remark}
%One can also obtain the well-known generating function for plane partitions \cite{Mac1995} using Remark \ref{RE1}:
%\begin{align}
%\sum_{\pi}z^{|\pi|}=\prod^{\infty}_{i=1}\frac{1}{(1-z^i)^i},
%\end{align}
%where the sum runs over plane partitions $\pi$ (cf. \cite{Whe2012}).
%\end{remark}
For $1\leq m\leq M$ and variables $\{x\}=\{x_1,\cdots, x_N\}$, $\{y\}=\{y_1,\cdots, y_N\}$, let %set
\begin{align*}
&a_m=\sum_{1\leq n\leq N}(-1)^{n-1}s_{(m,1^{n-1})}\{x\}b(n), \quad
a^*_m=\sum_{1\leq n\leq N}(-1)^{n-1}s_{(m,1^{n-1})}\{y\}c(-n), %A_m=\sum_{1\leq n\leq N}(-1)^{n-1}s_{(m,1^{n-1})}\{y\}c(-n),
\end{align*}
then
\begin{align}\label{bc4}
T_{kt}=\{a_k,a^*_t\}=\sum_{1\leq n\leq N}s_{(k,1^{n-1})}\{x\}s_{(t,1^{n-1})}\{y\}. %~T=(T_{\mu_i\mu_j})_{1\leq i,j\leq l}.
\end{align}
\par In view of Proposition \ref{e:genSchur} we have that %by (\ref{bc12}) and (\ref{bc13}) we have
\begin{align*}
&\mathbf{B}(y_1)\cdots \mathbf{B}(y_N)|0\rangle=\mathop{\overrightarrow\prod}\limits_{m\in [1,M]}(1+b(-m+1)a^*_m)|0\rangle,\\
&\langle 0|\mathbf{C}(x_N)\cdots\mathbf{C}(x_1)=\langle 0|\mathop{\overleftarrow\prod}\limits_{m\in [1,M]}(1+a_mc(m-1)).
\end{align*}
Therefore, we get that
\begin{proposition} For any positive integer $M$, let $T=(T_{ij})_{1\leq i,j\leq M}$, then
\begin{align}
\langle 0|\mathbf{C}(x_N)\cdots\mathbf{C}(x_1)\mathbf{B}(y_1)\cdots \mathbf{B}(y_N)|0\rangle=\det(I+T)=\sum_{\mu\subseteq[N,M]}s_\mu\{x\}s_\mu\{x\}.
\end{align}
\end{proposition}
%Put
%\begin{align}\label{bc4}
%T_{kt}=\{a_k,A_t\}=\sum_{1\leq n\leq N}s_{(k,1^{n-1})}\{x\}s_{(t,1^{n-1})}\{y\},~T=(T_{\mu_i\mu_j})_{1\leq i,j\leq l},
%\end{align}
\begin{proof} The result follows from the following simple calculation.
For strict partition $\mu=(\mu_1>\cdots >\mu_{l})$ inside the rectangle $[N, M]$,
let $T_{\mu}=(T_{\mu_i\mu_j})_{l\times l}$, then
\begin{align*}
&\langle 0|a_{\mu_{l}}c(\mu_{l}-1)\cdots a_{\mu_{1}}c(\mu_{1}-1)b(-\mu_{1}+1)a^*_{\mu_{1}}\cdots b(-\mu_{l}+1)a^*_{\mu_{l}}|0\rangle\\
=&\langle 0|a_{\mu_{l}}\cdots a_{\mu_{1}}a^*_{\mu_{1}}\cdots a^*_{\mu_{l}}|0\rangle=\det(T_{\mu}).
\end{align*}
\end{proof}

\section{Correlation functions of charged free bosons} %$\mathbf{B}^*(x)$ and $\mathbf{C}^*(x)$ of
In this section we define the charged free bosonic system and
 introduce two new operators $\mathbf{B}^{*}(x)$ and $\mathbf{C}^{*}(x)$. We will see that their correlation functions enjoy similar but distinct properties in view of the previous section. % $\langle 0|\mathbf{C}^{*}(x_N)\cdots\mathbf{C}^{*}(x_1)\mathbf{B}^{*}(y_1)\cdots \mathbf{B}^{*}(y_N)|0\rangle$.}
%In Section \ref{sub5}, we introduce operators $\mathbf{B}^{*}(x)$ and $\mathbf{C}^{*}(x)$ of charged free bosons and get some results which have close connection of $\mathbf{B}(x)$ and $\mathbf{C}(x)$. We also discuss the actions of $\mathbf{B}^{*}(x)$ on $\mathcal{M}^{(0)}$ and $\mathbf{C}^{*}(x)$ on $\mathcal{M}^{*(0)}$ in Section \ref{sub6}. In Section \ref{sub7}, we get correlation function $\langle 0|\mathbf{C}^{*}(x_N)\cdots\mathbf{C}^{*}(x_1)\mathbf{B}^{*}(y_1)\cdots \mathbf{B}^{*}(y_N)|0\rangle_{\beta\gamma}$ and find it is the inverse of $\langle 0|\mathbf{C}^{*}(x_N)\cdots\mathbf{C}^{*}(x_1)\mathbf{B}^{*}(y_1)\cdots \mathbf{B}^{*}(y_N)|0\rangle_{bc}$.
\subsection{Operators $\mathbf{B}^{*}(x)$ and $\mathbf{C}^{*}(x)$}\label{sub5}
Recall that the charged free bosons \cite{Wang1998,Li2011} are given by
\begin{align}
\varphi(z)=\sum_{i\in \mathbb{Z}}\varphi_{i}z^{-i-1},~~~~\varphi^{*}(z)=\sum_{i\in \mathbb{Z}}\varphi^{*}_{i}z^{-i}
\end{align}
with the commutation relations
\begin{align}\label{B3}
[\varphi_{i},\varphi^{*}_{j}]=\delta_{i,-j},~~~~[\varphi_{i},\varphi_{j}]=[\varphi^{*}_{i},\varphi^{*}_{j}]=0.
\end{align}
\par For a positive integer $M$, we define
\begin{align}
&\label{CB1}\mathbf{B}^{*}(x)=\exp\left(x\varphi_{-M}\varphi^*_{M-1}\right)\exp\left(x\varphi_{-M+1}\varphi^*_{M-2}\right)\cdots=\mathop{\overrightarrow\prod}\limits_{i\in (-\infty,M]}\exp\left(x\varphi_{-i}\varphi^*_{i-1}\right),\\
&\label{CB2}\mathbf{C}^{*}(x)=\cdots \exp\left(x\varphi_{-M+2}\varphi^*_{M-1}\right)\exp\left(x\varphi_{-M+1}\varphi^*_{M}\right)=\mathop{\overleftarrow\prod}\limits_{i\in (-\infty,M]}\exp\left(x\varphi_{-i+1}\varphi^*_{i}\right).
\end{align}

We first consider their basic properties and then study the actions on the Fock space and its dual in the next subsection.
%the completions $\widetilde{\mathcal{M}}^{(0)},\widetilde{\mathcal{M}}^{*(0)}$ respectively.

The following, similar to Proposition \ref{pro1}, can be proved in the same manner.
\begin{proposition} For a fixed $M\in\mathbb N$, we have that
\begin{align}
\mathbf{B}^{*}(x)=\exp\left(\sum^{\infty}_{n=1}\frac{(-1)^{n+1}}{n}x^n\bar{\wedge}_n\right),~~\mathbf{C}^{*}(x)=\exp\left(\sum^{\infty}_{n=1}\frac{(-1)^{n+1}}{n}x^n\bar{\wedge}^*_n\right),
\end{align}
where $\bar{\wedge}_n=\sum^M_{-\infty}\varphi_{-i}\varphi^*_{i-n}$ and $\bar{\wedge}^*_n=\sum^M_{-\infty}\varphi_{-i+n}\varphi^*_{i}.$
\end{proposition}
%\begin{proof}
%As we did in the proof of Proposition \ref{pro1}.
%\end{proof}
\begin{proposition} For any fixed $M$, we also have that
\begin{align}\label{boson1}
\mathbf{B}^{*}(x)\mathbf{B}^{*}(y)=\mathbf{B}^{*}(y)\mathbf{B}^{*}(x),~~~~\mathbf{C}^{*}(x)\mathbf{C}^{*}(y)=\mathbf{C}^{*}(y)\mathbf{C}^{*}(x).
\end{align}
\end{proposition}
\begin{proposition}
Setting $\lim\limits_{M\rightarrow \infty}\bar{\wedge}_n=h_{-n},~\lim\limits_{M\rightarrow \infty}\bar{\wedge}^*_n=h_n$, then
\begin{gather}\label{eq3}
[h_m,h_n]=-m\delta_{m,-n},\\
\lim\limits_{M\rightarrow \infty}\mathbf{B}^{*}(x)=\exp\left(\sum^{\infty}_{n=1}\frac{(-1)^{n+1}}{n}x^nh_{-n}\right),~~\lim\limits_{M\rightarrow \infty}\mathbf{C}^{*}(x)=\exp\left(\sum^{\infty}_{n=1}\frac{(-1)^{n+1}}{n}x^nh_n\right),\\
\lim\limits_{M\rightarrow \infty}\mathbf{C}^{*}(x)\mathbf{B}^{*}(y)=(1-xy)\lim\limits_{M\rightarrow \infty}\mathbf{B}^{*}(y)\mathbf{C}^{*}(x).
\end{gather}
\end{proposition}
\subsection{Actions of $\mathbf{B}^{*}(x)$ on $\widetilde{\mathcal{M}}^{(0)}$ and $\mathbf{C}^{*}(x)$ on $\widetilde{\mathcal{M}}^{*(0)}$}\label{sub6}
Let $\mathcal{M}$ (resp. $\mathcal{M}^*$) be the Fock space of the charged free bosons generated by monomials in the bosons
$\varphi_i, \varphi^*_i$ with their actions on the vacuum vector $|0\rangle$ (resp. $\langle0|$) defined by
\begin{align}\label{vac1}
\varphi_i|0\rangle=0,~~\varphi^{*}_{i+1}|0\rangle=0,~~i\geq 0 ~~(\text{resp.}~\langle0|\varphi_i=\langle0|\varphi^{*}_{i+1}=0,~~i<0).
\end{align}
\par As in section \ref{sub2}, we define the subspace $\mathcal{M}^{(0)}$ of $\mathcal{M}$ and subspace $\mathcal{M}^{*(0)}$ of $\mathcal{M}^*$ with bases
\begin{align*}
&\text{Basis}(\mathcal{M}^{(0)})=\left\{\varphi^{n_l}_{-l}\dots\varphi^{n_1}_{-1}\varphi^{*n_{-k}}_{-k}\dots\varphi^{*n_0}_{0}|0\rangle|\sum^{k}_{i=0}n_{-i}=\sum^{l}_{j=1}n_j, n_i\geq 0\right\},\\
&\text{Basis}(\mathcal{M}^{*(0)})=\left\{\langle0|\varphi^{m_0}_{0}\dots\varphi^{m_{-k}}_{k}\varphi^{*m_1}_{1}\dots\varphi^{*m_l}_{l}|
\sum^{k}_{i=0}m_{-i}=\sum^{l}_{j=1}m_j, m_i\geq 0\right\}.
\end{align*}
\par Let $\widetilde{\mathcal{M}},\widetilde{\mathcal{M}}^{*},\widetilde{\mathcal{M}}^{(0)},\widetilde{\mathcal{M}}^{*(0)}$ be the completion of $\mathcal{M},\mathcal{M}^{*},\mathcal{M}^{(0)},\mathcal{M}^{*(0)}$ (by gradation) respectively.
\par  An element $\tau\in \widetilde{\mathcal{M}}$ is called a bosonic KP tau function\cite{Kv1990,Li2011} if $\tau$ satisfies
\begin{align}\label{hirota1}
\mathrm{Res}_{z}\varphi(z)\otimes \varphi^{*}(z)(\tau\otimes \tau)=0.
\end{align}
It is known that \eqref{hirota1} has only trivial solutions in $\mathcal{M}$ \cite{JL2019}.

\begin{remark} The bosonic CKP tau functions are also related to bosonic KP tau functions
\cite{VOS2012} (see also \cite{DJKM6,Kv1990}).
%The authors in \cite{VOS2012} used the bosonic KP tau function based on symplectic bosons to study the bosonic CKP tau function\cite{DJKM6}.
\end{remark}

%\begin{remark}
It follows from \eqref{vac1} that $\mathrm{Res}_{z}\varphi(z)\otimes \varphi^{*}(z)(|0\rangle\otimes |0\rangle)=0$, and
by \cite{JL2019,Wang1998} we have that
\begin{align*}
%\mathrm{Res}_{z}\varphi(z)\otimes \varphi^{*}(z)(|0\rangle\otimes |0\rangle)=0,~~
[\mathrm{Res}_{z}\varphi(z)\otimes \varphi^{*}(z),\mathbf{B}^{*}(x)\otimes \mathbf{B}^{*}(x)]=0,
\end{align*}
thus $\mathbf{B}^{*}(x_1)\mathbf{B}^{*}(x_2)\cdots \mathbf{B}^{*}(x_N)|0\rangle$ are bosonic KP tau functions for all
$x_i$, whose coefficients are KP tau functions in the completion $\widetilde{\mathcal{M}}$.
\begin{theorem}\label{th1}
For any $M$ and $\{x\}=\{x_1,\cdots, x_N\}$, we have
\begin{align}\label{e:taufcn}
\mathbf{B}^{*}(x_1)\cdots \mathbf{B}^{*}(x_N)|0\rangle=\exp\left(\sum_{\substack{1\leq m\leq M \\ 1\leq n\leq N}}(-1)^{m-1}s_{(m,1^{n-1})}\{x\}\varphi_{-m}\varphi^*_{1-n}\right)|0\rangle.
\end{align}
%$s_{(m,1^{n-1})}(x)$ is the hook Schur function associated to partition $\{m,\underbrace{1\cdots 1}_{n-1}\}$ in the variables $x_1,x_2,\cdots x_{N}$.
\end{theorem}
\begin{proof}
A special case of the BCH theorem says that if $[A,B]$ commutes with both $A$ and $B$, then
\begin{align*}
\exp(A)\exp(B)=\exp(B+[A, B])\exp(A).
\end{align*}
\par Using this formula to successively move $e^{x_1\varphi_{-M}\varphi^*_{M-1}}, e^{x_1\varphi_{-M+1}\varphi^*_{M-2}}, \dots$ to the right and
noting that $\varphi^*_i|0\rangle=0 (i>0)$, we have that %repeatedly we have that
\begin{align*}
\mathbf{B}^{*}(x_1)|0\rangle&=\exp\left(x_1\varphi_{-M}\varphi^*_{M-1}\right)\exp\left(x_1\varphi_{-(M-1)}\varphi^*_{M-2}\right)\mathop{\overrightarrow\prod}\limits_{i\in [1,M-2]}\exp\left(x_1\varphi_{-i}\varphi^*_{i-1}\right)|0\rangle\\
&=\exp\left(x_1(\varphi_{-(M-1)}-x_1\varphi_{-M})\varphi^*_{M-2}\right)\exp\left(x_1\varphi_{-M}\varphi^*_{M-1}\right)\mathop{\overrightarrow\prod}\limits_{i\in [1,M-2]}\exp\left(x_1\varphi_{-i}\varphi^*_{i-1}\right)|0\rangle\\
&=\exp\left((x_1\varphi_{-(M-1)}-x^2_1\varphi_{-M})\varphi^*_{M-2}\right)\mathop{\overrightarrow\prod}\limits_{i\in [1,M-2]}\exp\left(x_1\varphi_{-i}\varphi^*_{i-1}\right)|0\rangle\\
%&=\exp\left((x_1\varphi_{-(M-1)}-x^2_1\varphi_{-M})\varphi^*_{M-2}\right)\exp\left(x_1\varphi_{-(M-2)}\varphi^*_{M-3}\right)\mathop{\overrightarrow\prod}\limits_{i\in [1,M-3]}\exp\left(x_1\varphi_{-i}\varphi^*_{i-1}\right)|0\rangle\\
&=\exp\left((x_1\varphi_{-(M-2)}-x^2_1\varphi_{-(M-1)}+x^3_1\varphi_{-M})\varphi^*_{M-3}\right)\mathop{\overrightarrow\prod}\limits_{i\in [1,M-3]}\exp\left(x_1\varphi_{-i}\varphi^*_{i-1}\right)|0\rangle=\cdots\\
&=\exp\left(\sum_{1\leq m\leq M }(-1)^{m-1}s_{m}\{x_1\}\varphi_{-m}\varphi^*_0\right)|0\rangle,
\end{align*}
\par Assuming \eqref{e:taufcn} holds for $N-1$ and $\bar{x}=\{x_1, \dots, x_{N-1}\}$, we have that
%\begin{align}
%\mathbf{B}^{*}(x_1)\cdots \mathbf{B}^{*}(x_{N-1})|0\rangle=\exp\left(\sum_{\substack{1\leq m\leq M \\ 1\leq n\leq N-1}}(-1)^{m-1}s_{(m,1^{n-1})}\{\bar{x}\}\varphi_{-m}\varphi^*_{1-n}\right)|0\rangle,~ \{\bar{x}\}=\{x_1,\cdots, x_{N-1}\},
%\end{align}
%then
\begin{align*}
\mathbf{B}^{*}(x_1)&\cdots \mathbf{B}^{*}(x_N)|0\rangle=\mathbf{B}^{*}(x_N)\mathbf{B}^{*}(x_1)\cdots \mathbf{B}^{*}(x_{N-1})|0\rangle\\
%=&\mathop{\overrightarrow\prod}\limits_{i\in [2-N,M]}exp(x_N\varphi_{-i}\varphi^*_{i-1})\mathbf{B}(x_1)\cdots \mathbf{B}(x_{N-1})|0\rangle\\
=&\mathop{\overrightarrow\prod}\limits_{i\in [2-N,M]}\exp\left(x_N\varphi_{-i}\varphi^*_{i-1}\right)\exp\left(\sum_{\substack{1\leq m\leq M \\ 1\leq n\leq N-1}}(-1)^{m-1}s_{(m,1^{n-1})}\{\bar{x}\}\varphi_{-m}\varphi^*_{1-n}\right)|0\rangle\\
=&\mathop{\overrightarrow\prod}\limits_{i\in [3-N,M]}\exp\left(x_N\varphi_{-i}\varphi^*_{i-1}\right)\exp\left(\sum_{\substack{1\leq m\leq M \\ 1\leq n\leq N-2}}(-1)^{m-1}s_{(m,1^{n-1})}\{\bar{x}\}\varphi_{-m}\varphi^*_{1-n}\right)\\
&\exp\left(x_N\varphi_{N-2}\varphi^*_{1-N}\right)\exp\left(\sum_{1\leq m\leq M }(-1)^{m-1}s_{(m,1^{N-2})}\{\bar{x}\}\varphi_{-m}\varphi^*_{2-N}\right)|0\rangle\\
=&\mathop{\overrightarrow\prod}\limits_{i\in [3-N,M]}\exp\left(x_N\varphi_{-i}\varphi^*_{i-1}\right)\exp\left(\sum_{\substack{1\leq m\leq M \\ 1\leq n\leq N-2}}(-1)^{m-1}s_{(m,1^{n-1})}\{\bar{x}\}\varphi_{-m}\varphi^*_{1-n}\right)\\
&\exp\left(\sum_{1\leq m\leq M }(-1)^{m-1}s_{(m,1^{N-2})}\{\bar{x}\}\varphi_{-m}(\varphi^*_{2-N}+x_N\varphi^*_{1-N})\right)|0\rangle\\
%=&\cdots\\
%=&\mathop{\overrightarrow\prod}\limits_{i\in [1,M]}exp(x_N\varphi_{-i}\varphi^*_{i-1})exp(\sum_{\substack{1\leq m\leq M \\ 1\leq n\leq N-1}}(-1)^{m-1}s_{(m,1^{n-1})}\{\bar{x}\}\varphi_{-m}(\varphi^*_{1-n}+x_N\varphi^*_{-n}))|0\rangle\\
=&\mathop{\overrightarrow\prod}\limits_{i\in [1,M]}\exp\left(x_N\varphi_{-i}\varphi^*_{i-1}\right)\exp\left(\sum_{\substack{1\leq m\leq M \\ 1\leq n\leq N-1}}(-1)^{m-1}s_{(m,1^{n-1})}\{\bar{x}\}\varphi_{-m}(\varphi^*_{1-n}+x_N\varphi^*_{-n})\right)|0\rangle.
\end{align*}
%The following  procedure used (\ref{bc3}).
\par Denote the argument of the second exp by $P=\sum_{1\leq m\leq M}P_m$, where
%$\sumFor $1\leq m\leq M$, let
\begin{align*}
%&P=\sum_{\substack{1\leq m\leq M \\ 1\leq n\leq N-1}}(-1)^{m-1}s_{(m,1^{n-1})}\{\bar{x}\}\varphi_{-m}(\varphi^*_{1-n}+x_N\varphi^*_{-n}),\\
&P_m=\sum_{1\leq n\leq N-1}(-1)^{m-1}s_{(m,1^{n-1})}\{\bar{x}\}\varphi_{-m}(\varphi^*_{1-n}+x_N\varphi^*_{-n}).
\end{align*}

We consider $Q_1=x_N\varphi_{-1}\varphi^*_0+P_1$ and $Q_j=[x_N\varphi_{-j}\varphi^*_{j-1},Q_{j-1}]+P_j,~2\leq j\leq M$. Then by (\ref{bc3}) and (\ref{bc11}), we have
\begin{align*}
Q_1=&\left(s_{(1)}\{\bar{x}\}+x_N\right)\varphi_{-1}\varphi^*_{0}+\sum_{2\leq n\leq N-1}\left(s_{(1^n)}\{\bar{x}\}+s_{(1^{n-1})}\{\bar{x}\}x_N\right)\varphi_{-1}\varphi^*_{1-n}+s_{(1^{N-1})}\{\bar{x}\}x_N\varphi_{-1}\varphi^*_{1-N}\\
=&\sum_{1\leq n\leq N}(s_{(1^n)}\{\bar{x}\}+s_{(1^{n-1})}\{\bar{x}\}x_N)\varphi_{-1}\varphi^*_{1-n}\\
%=&s_{(1)}(x)\varphi_{-1}\varphi^*_{0}+\sum_{2\leq n\leq N-1}s_{(1^n)}(x)\varphi_{-1}\varphi^*_{1-n}+s_{(1^N)}(x)\varphi_{-1}\varphi^*_{1-N}\\
=&\sum_{1\leq n\leq N}s_{(1,1^{n-1})}\{x\}\varphi_{-1}\varphi^*_{1-n}.
\end{align*}
Similarly we have
\begin{align*}
Q_2&=[x_N\varphi_{-2}\varphi^*_1,Q_1]+P_2=-\sum_{1\leq n\leq N}\left(s_{(1^n)}\{\bar{x}\}x_N+s_{(1^{n-1})}\{\bar{x}\}x^2_N\right)\varphi_{-2}\varphi^*_{1-n}\\
&-\sum_{1\leq n\leq N-1}s_{(2,1^{n-1})}\{\bar{x}\}\varphi_{-2}\varphi^*_{1-n}
-\sum_{2\leq n\leq N}s_{(2,1^{n-2})}\{\bar{x}\}x_N\varphi_{-2}\varphi^*_{1-n}\\
=&-\left(s_0\{\bar{x}\}x^2_N+s_1\{\bar{x}\}x_N+s_2\{\bar{x}\}\right)\varphi_{-2}\varphi^*_0\\
&-\sum_{2\leq n\leq N-1}\left(s_{(1^{n-1})}\{\bar{x}\}x^2_N+s_{(1^n)}\{\bar{x}\}x_N+s_{(2,1^{n-2})}\{\bar{x}\}x_N+s_{(2,1^{n-1})}\{\bar{x}\}\right)\varphi_{-2}\varphi^*_{1-n}\\
&-\left(s_{(1^{N-1})}\{\bar{x}\}x^2_N+s_{(1^N)}\{\bar{x}\}x_N+s_{(2,1^{N-2})}\{\bar{x}\}x_N\right)\varphi_{-2}\varphi^*_{1-N}\\
=&-\sum_{1\leq n\leq N}s_{(2,1^{n-1})}\{x\}\varphi_{-2}\varphi^*_{1-n}.
\end{align*}

Continuing in this way, we have that % the calculation, we can get
\begin{align*}
Q_j=(-1)^{j-1}\sum_{1\leq n\leq N}s_{(j,1^{n-1})}\{x\}\varphi_{-j}\varphi^*_{1-n},~~~1\leq j\leq M.
\end{align*}
\par Now we compute by using the BCH formula:
\begin{align*}
\mathbf{B}^{*}(x_1)&\cdots \mathbf{B}^{*}(x_N)|0\rangle=\mathop{\overrightarrow\prod}\limits_{i\in [1,M]}\exp\left(x_N\varphi_{-i}\varphi^*_{i-1}\right)\exp\left(\sum_{1\leq j\leq M}P_j\right)|0\rangle\\
=&\mathop{\overrightarrow\prod}\limits_{i\in [2,M]}\exp\left(x_N\varphi_{-i}\varphi^*_{i-1}\right)\exp\left(\sum_{1\leq j\leq M}P_j+x_N\varphi_{-1}\varphi^*_{0}\right)|0\rangle\\
=&\mathop{\overrightarrow\prod}\limits_{i\in [3,M]}\exp\left(x_N\varphi_{-i}\varphi^*_{i-1}\right)\exp\left(x_N\varphi_{-2}\varphi^*_1\right)\exp\left(Q_1+\sum_{2\leq j\leq M}P_j\right)|0\rangle\\
=&\mathop{\overrightarrow\prod}\limits_{i\in [3,M]}\exp\left(x_N\varphi_{-i}\varphi^*_{i-1}\right)\exp\left(Q_1+Q_2+\sum_{3\leq j\leq M}P_j\right)|0\rangle=\cdots\\
=&\exp\left(\sum_{1\leq j\leq M}Q_j\right)|0\rangle=\exp\left(\sum_{\substack{1\leq m\leq M \\ 1\leq n\leq N}}(-1)^{m-1}s_{(m,1^{n-1})}\{x\}\varphi_{-m}\varphi^*_{1-n}\right)|0\rangle.
\end{align*}
\end{proof}
\begin{corollary}\label{cor1} For $\{x\}=\{x_1,\cdots, x_{N}\}$ and fixed $M$,  we also have the expansion
\begin{align}
\langle 0|\mathbf{C}^{*}(x_N)\cdots\mathbf{C}^{*}(x_1)=\langle 0|\exp\left(\sum_{\substack{1\leq m\leq M \\ 1\leq n\leq N}}(-1)^{m-1}s_{(m,1^{n-1})}\{x\}\varphi_{n-1}\varphi^*_{m}\right).
\end{align}
\end{corollary}
\subsection{Correlation function $\langle 0|\mathbf{C}^{*}(x_N)\cdots\mathbf{C}^{*}(x_1)\mathbf{B}^{*}(y_1)\cdots \mathbf{B}^{*}(y_N)|0\rangle$}\label{sub7}
Setting $\langle0|1|0\rangle=1$, for $\langle \uline{m}|=\langle0|\varphi^{m_0}_{0}\dots\varphi^{m_{-k}}_{k}\varphi^{*m_1}_{1}\dots\varphi^{*m_l}_{l}$ and $|\uline{n}\rangle=\varphi^{n_l}_{-l}\dots\varphi^{n_1}_{-1}\varphi^{*n_{-k}}_{-k}\dots\varphi^{*n_0}_{0}|0\rangle$, we have
\begin{align}\label{B2}
\langle \uline{m}|\uline{n}\rangle=(-1)^{\sum^l_{i=1}n_i}\prod^l_{s=-k}\delta_{m_s,n_s}m_s!.
\end{align}
\begin{theorem}\label{th2}
For any positive integer $M$ and $\{x\}=\{x_1,\cdots, x_{N}\}$, $\{y\}=\{y_1,\cdots, y_{N}\}$, one has that
\begin{align}\label{eq2}
\langle 0|\mathbf{C}^{*}(x_N)\cdots\mathbf{C}^{*}(x_1)\mathbf{B}^{*}(y_1)\cdots \mathbf{B}^{*}(y_N)|0\rangle=\frac{1}{\sum\limits_{\mu\subseteq[N,M]}s_\mu\{x\}s_\mu\{y\}}.
\end{align}
In particular, when $M\rightarrow\infty$, we get another proof of the Cauchy identity:
\begin{align}\label{cauchy2}
\langle 0|\mathbf{C}^{*}(x_N)\cdots\mathbf{C}^{*}(x_1)\mathbf{B}^{*}(y_1)\cdots \mathbf{B}^{*}(y_N)|0\rangle=\prod_{i, j=1}^N(1-x_iy_j)=\frac{1}{\sum\limits_{l(\mu)\leq N}s_\mu\{x\}s_\mu\{y\}}.
\end{align}
\end{theorem}
The proof is contained in Appendix A.

Comparing the right sides of (\ref{eq1}) and (\ref{eq2}), we have
\begin{corollary}
For a positive integer $M$,
\begin{align}
\langle 0|\mathbf{C}^{*}(x_N)\cdots\mathbf{C}^{*}(x_1)\mathbf{B}^{*}(y_1)\cdots \mathbf{B}^{*}(y_N)|0\rangle=\frac{1}{\langle 0|\mathbf{C}(x_N)\cdots\mathbf{C}(x_1)\mathbf{B}(y_1)\cdots \mathbf{B}(y_N)|0\rangle} .
\end{align}
\end{corollary}
%\par Now we have another proof of the famous Cauchy formula \cite{Mac1995}:
\begin{corollary}
Taking the limit $M\rightarrow \infty,N\rightarrow \infty$ we get that
\begin{align}
\lim_{M, N\rightarrow \infty}\langle 0|\mathbf{C}^{*}(x_N)\cdots\mathbf{C}^{*}(x_1)\mathbf{B}^{*}(y_1)\cdots \mathbf{B}^{*}(y_N)|0\rangle=\prod^{\infty}_{i,j=1}(1-x_iy_i)=\frac{1}{\sum\limits_\mu s_\mu(x)s_\mu(y)},
\end{align}
where the sum is over all partitions $\mu$. % and $\{x\}=\{x_1,\cdots, x_{N}\}$, $\{y\}=\{y_1,\cdots, y_{N}\}$.
In particular,
%\begin{align*}
%x_i=y_i=z^{i-\frac{1}{2}}
%\end{align*}
%gives
\begin{align*}
\lim_{M, N\rightarrow \infty}\langle 0|\mathbf{C}^{*}(z^{N-\frac{1}{2}})\cdots\mathbf{C}^{*}(z^{1-\frac{1}{2}})\mathbf{B}^{*}(z^{1-\frac{1}{2}})\cdots \mathbf{B}^{*}(z^{N-\frac{1}{2}})|0\rangle=\prod^{\infty}_{i=1}
(1-z^i)^i.
\end{align*}
\end{corollary}
%\begin{remark}
%If one has an exact presentation of $\mathbf{B}^{*}(x)$ (resp. $\mathbf{C}^{*}(x)$) on $\widetilde{\mathcal{M}}^{(0)}$ (resp. $\widetilde{\mathcal{M}}^{*(0)}$), then one can get a generating function of $\prod^{\infty}_{i=1}
%(1-z^i)^i$.
%\end{remark}
\section{Neutral fermions}
This section is parallel to Section \ref{se1}. %and the results are obtained by the similar method.
It is known that the usual KP tau functions can be formulated in terms of Schur functions, which are integral
combinations of the power sum $p_r=\sum_ix_i^r$. There is another family of well-known symmetric functions called Schur's Q-function $Q_{\nu}$, which
are orthogonal rational linear combination of the odd degree
power sum symmetric function $p_{2r-1}$. It turns out that the BKP tau functions can be realized by the Schur Q-functions \cite{You1990}.
%\uline{We will give a new derivation of this beautiful result and provide a new
%interpretation of the BKP tau functions using the neutral fermions}
Foda and Wheeler\cite{FWZ2007} used the QISM to show Bethe eigenvectors of the $i$-boson model are BKP tau functions (see Appendix C), we aim to give a new derivation of this beautiful result.

We start with the neutral fermions defined by infinite operators $\{\phi_m\}_{m\in \mathbb{Z}}$ with the relations
\begin{align}
\{\phi_m,\phi_n\}=2(-1)^m\delta_{m+n,0}.
\end{align}
\par The Fock space $\mathcal{F}_{\phi}$ and its dual $\mathcal{F}^{*}_{\phi}$ are the vector spaces generated by $|0\rangle$ and $\langle0|$ subject to
\begin{align}
\phi_m|0\rangle=\langle0|\phi_{-m}=0,~~m<0.
\end{align}

One can decompose $\mathcal{F}_{\phi}$ and $\mathcal{F}^*_{\phi}$ into a direct sum
\begin{align*}
\mathcal{F}_{\phi}=\bigoplus_{i\in\{0,1\}}\mathcal{F}^{(i)}_{\phi},~~\mathcal{F}^*_{\phi}=\bigoplus_{i\in\{0,1\}}\mathcal{F}^{*(i)}_{\phi},
\end{align*}
where $\mathcal{F}^{(i)}_{\phi}$ and $\mathcal{F}^{*(i)}_{\phi}$ are the subspaces generated by the action of all neutral fermion monomials of length $i \mod 2$ on $|0\rangle$ and $\langle0|$, respectively.
For fixed $M$, we introduce
\begin{align}
\label{ne1}\mathbf{\tilde{B}}(x)=&\exp\left(\frac{1}{2}x(-1)^M\phi_{-M+1}\phi_M\right)\exp\left(\frac{1}{2}x(-1)^{M-1}\phi_{-M+2}\phi_{M-1}\right)\cdots\\
\notag&\exp\left(\frac{1}{2}x(-1)^{-M+2}\phi_{M-1}\phi_{-M+2}\right)\exp\left(\frac{1}{2}x(-1)^{-M+1}\phi_M\phi_{-M+1}\right)\\
\notag=&\mathop{\overrightarrow\prod}\limits_{j\in [-M+1,M]}\exp\left(\frac{1}{2}x(-1)^j\phi_{-j+1}\phi_j\right),\\
\label{ne2}\mathbf{\tilde{C}}(x)=&\exp\left(\frac{1}{2}x(-1)^{-M}\phi_{M-1}\phi_{-M}\right)\exp\left(\frac{1}{2}x(-1)^{-M+1}\phi_{M-2}\phi_{-M+1}\right)\cdots \\ \notag&\exp\left(\frac{1}{2}x(-1)^{M-2}\phi_{-M+1}\phi_{M-2}\right)\exp\left(\frac{1}{2}x(-1)^{M-1}\phi_{-M}\phi_{M-1}\right)\\
\notag=&\mathop{\overleftarrow\prod}\limits_{i\in [-M+1,M]}\exp\left(\frac{1}{2}x(-1)^{i-1}\phi_{-i}\phi_{i-1}\right).
\end{align}
\par The vector $\tau\in \mathcal{F}_{\phi}$  is called a BKP tau function\cite{KRR2013,Whe2012} if
\begin{align*}
\sum_{i\in \mathbb{Z}}(-1)^{i}\phi_i\otimes \phi_{-i}(\tau\otimes \tau)=\phi_0\tau\otimes \phi_0\tau.
\end{align*}

It is clear that
$\mathbf{\tilde{B}}(x_1)\cdots\mathbf{\tilde{B}}(x_N)|0\rangle$ is a BKP tau function.
%\uline{ Using the same method of Section \ref{se1}, the action of $\mathbf{\tilde{B}}(x)$ can be seen to agree with that of $\mathbb{\tilde{B}}(x)$ on $\mathcal{F}^{(0)}_{\phi}$, thus $\mathbf{\tilde{B}}(x_1)\cdots\mathbf{\tilde{B}}(x_N)|0\rangle$ can be called a Bethe eigenvector.}{\color{red}delete}
Using the similar method of Prop. \ref{pro1}, we have
\begin{proposition} For any fixed $M\in \mathbb N$, one has that
\begin{align}
\mathbf{\tilde{B}}(x)=\exp\left(\sum_{n\in 2\mathbb{N}+1}\frac{1}{n}x^n\Lambda_n\right),~~\mathbf{\tilde{C}}(x)=\exp\left(\sum_{n\in 2\mathbb{N}+1}\frac{1}{n}x^n\Lambda^*_n\right),
\end{align}
where
\begin{align}
\Lambda_n=\frac{1}{2}\sum^M_{-M+n}(-1)^j\phi_{-j+n}\phi_j,~~\Lambda^*_n=\frac{1}{2}\sum^{M-n}_{-M}(-1)^j\phi_{-j-n}\phi_j.
\end{align}
\end{proposition}
\begin{remark}\label{RE2}
Under $M\rightarrow\infty$, let $\lambda_n=\Lambda^*_n, ~\lambda_{-n}=\Lambda_n,~n\geq 1$, we have that
\begin{align}\label{e:comm}
[\lambda_m,\lambda_n]=2m\delta_{m,-n}.
\end{align}
\end{remark}
The following result is clear.
\begin{proposition} For any fixed $M$, we have that
\begin{align}
\mathbf{\tilde{B}}(x)\mathbf{\tilde{B}}(y)=\mathbf{\tilde{B}}(y)\mathbf{\tilde{B}}(x),~~~~\mathbf{\tilde{C}}(x)\mathbf{\tilde{C}}(y)=\mathbf{\tilde{C}}(y)\mathbf{\tilde{C}}(x).
\end{align}
\end{proposition}
A partition $\tilde{\nu}=(\tilde{\nu}_1, \ldots, \tilde{\nu}_l)$ is called {\it strict} if $\tilde{\nu}_1>\cdots>\tilde{\nu}_l$. Define the basis vectors
\begin{equation}\label{ch4-eq1}
|\tilde{\nu})=\begin{cases}
\phi_{\tilde{\nu}_1}\phi_{\tilde{\nu}_2}\cdots \phi_{\tilde{\nu}_{\ell}}|0\rangle ~~&~~ \ell\ \mathrm{even}\\
\phi_{\tilde{\nu}_1}\phi_{\tilde{\nu}_2}\cdots \phi_{\tilde{\nu}_{\ell}}\phi_0|0\rangle ~~&~~\ell\ \mathrm{odd},
\end{cases}
\end{equation}
\begin{equation}\label{ch4-eq2}
(\tilde{\nu}|=\begin{cases}
\langle0|\phi_{-\tilde{\nu}_1}\phi_{-\tilde{\nu}_2}\cdots \phi_{-\tilde{\nu}_{\ell}} ~~&~~ \ell\ \mathrm{even}\\
\langle0|\phi_0\phi_{-\tilde{\nu}_1}\phi_{-\tilde{\nu}_2}\cdots \phi_{-\tilde{\nu}_{\ell}} ~~&~~\ell\ \mathrm{odd}.
\end{cases}
\end{equation}

Recall that the Schur $Q$-function $Q_{\tilde{\nu}}$ associated to the strict partition $\tilde{\nu}$ \cite{Mac1995} is defined by
$$Q_{\tilde{\nu}}(x_1, \cdots, x_n)%=Q_{(\la_1,\cdots,\la_l)}
=2^{l}\sum_{w\in S_n/S_{n-l}}x_{w(1)}^{\tilde{\nu}_1}\cdots
x_{w(n)}^{\tilde{\nu}_n} \prod_{i<j}\frac{x_{w(i)}+x_{w(j)}}{x_{w(i)}-x_{w(j)}},
$$
where $n\geq l=l(\tilde{\nu})$, $S_n$ acts on $x_1,\cdots, x_n$ and $S_{n-l}$ acts on the last $n-l$ variables.
\begin{theorem}\label{th5}  For any fixed $M$ and strict partition $\tilde{\nu}$ of length $l(\tilde{\nu})$, one has that
\begin{align}
\mathbf{\tilde{B}}(x)|\tilde{\nu})&=\sum_{\tilde{\nu}\prec \tilde{\mu}}2^{\#(\tilde{\mu}|\tilde{\nu})-l(\tilde{\mu})+l(\tilde{\nu})}x^{|\tilde{\mu}|-|\tilde{\nu}|}|\tilde{\mu}),\\
(\tilde{\nu}|\mathbf{\tilde{C}}(x)&=\sum_{\tilde{\nu}\prec \tilde{\mu}}2^{\#(\tilde{\mu}|\tilde{\nu})-l(\tilde{\mu})+l(\tilde{\nu})}x^{|\tilde{\mu}|-|\tilde{\nu}|}(\tilde{\mu}|,
\end{align}
where $\#(\tilde{\mu}|\tilde{\nu})$ denotes the number of parts in $\tilde{\mu}$ that are not in $\tilde{\nu}$.
\end{theorem}

Our result is summarized as follows.
\begin{theorem} For a fixed $M\in\mathbb N$, $\{x\}=\{x_1,\cdots, x_{N}\}$ and $\{y\}=\{y_1,\cdots, y_{N}\}$, we have that
\begin{align}
&\mathbf{\tilde{B}}(x_1)\cdots\mathbf{\tilde{B}}(x_N)|0\rangle=\sum_{\tilde{\mu}\in [N,M]}2^{-\ell(\tilde{\mu})}Q_{\tilde{\mu}}\{x\}|\tilde{\mu}),\\
&\langle 0|\mathbf{\tilde{C}}(x_N)\cdots\mathbf{\tilde{C}}(x_1)=\sum_{\tilde{\mu}\in [N,M]}2^{-\ell(\tilde{\mu})}Q_{\tilde{\mu}}\{x\}(\tilde{\mu}|,\\
\label{eq6}&\langle 0|\mathbf{\tilde{C}}(x_N)\cdots\mathbf{\tilde{C}}(x_1)\mathbf{\tilde{B}}(y_1)\cdots\mathbf{\tilde{B}}(y_N)|0\rangle=
\sum_{\tilde{\mu}\in [N,M]}2^{-\ell(\tilde{\mu})}Q_{\tilde{\mu}}\{x\}Q_{\tilde{\mu}}\{y\},
\end{align}
where $\tilde{\mu}$ runs through all strict partitions inside the box $[N, M]$ of height $N$ and width $M$. %and $Q_{\tilde{\mu}}\{x\}$ is the $Schur~ Q$-$function$ \cite{Mac1995}.
\end{theorem}
\begin{remark} The last identity \eqref{eq6} can also be realized as hypergeometric tau functions in \cite[Eq. (1.2.9)]{Or2003}.
\end{remark}
\begin{remark} Using the Heisenberg relation \eqref{e:comm}, it is easy to derive that under $M\rightarrow\infty$
\begin{align}
\langle 0|\mathbf{\tilde{C}}(x_N)\cdots\mathbf{\tilde{C}}(x_1)\mathbf{\tilde{B}}(y_1)\cdots\mathbf{\tilde{B}}(y_N)|0\rangle
=\prod^N_{i,j=1}\frac{1+x_iy_j}{1-x_iy_j},
\end{align}
which together with \eqref{eq6} gives another proof of the well-known Schur's identity \cite{Mac1995}:
\begin{align}
\sum_{l(\tilde{\mu})\leq N}2^{-\ell(\tilde{\mu})}Q_{\tilde{\mu}}\{x\}Q_{\tilde{\mu}}\{y\}=\prod^N_{i,j=1}\frac{1+x_iy_j}{1-x_iy_j},
\end{align}
where $\tilde{\mu}$ runs through all strict partitions with length $\leq N$, $\{x\}=\{x_1,\cdots, x_{N}\}$, and $\{y\}=\{y_1,\cdots, y_{N}\}$. %\cite{Mac1995,Whe2012}).
%Which can also be found in \cite{Mac1995,Whe2012}.
\end{remark}
%\begin{remark}
%$\mathbf{\tilde{B}}(x)$ and $\mathbf{\tilde{C}}(x)$ can be seen as the realization of $\mathbb{\tilde{B}}(x)$ and $\mathbb{\tilde{C}}(x)$ in neutral fermions in some ways. How to construct the corresponding operators in charged $t$-fermions in relation with $q$-boson model is worthy of further investigation.
%\end{remark}
\begin{appendix}
\section{} This appendix gives a detailed proof of Theorem \ref{th2}. We start with an easy lemma.
\begin{lemma}\label{le1}
The constant term of the function $(1-\frac{a_1}{z})^{-1}\cdots (1-\frac{a_n}{z})^{-1}(1-bz)^{-1}$
is $(1-a_1b)^{-1}\cdots (1-a_nb)^{-1}$, and the constant term of
$(c_1+\frac{a_1}{z})^{-1}\cdots (c_n+\frac{a_n}{z})^{-1}(1+bz)^{-1}$ is
$(c_1-a_1b)^{-1}\cdots (c_n-a_nb)^{-1}$.
\end{lemma}
\begin{proposition}\footnote{Professor Stanley informed us that the proposition may be seen as a special case of the MacMahon Master Theorem.
 We provide a direct proof for completeness.}
Let $T=(T_{ij})_{1\leq i,j\leq n}$ and $Z_{k}=\sum_{1\leq i\leq n}T_{ik}\frac{z_k}{z_i}~(1\leq k\leq n)$. Then the constant term $CT(A)$ of the multivariate function $A$ in the $z_i: A=\prod^n_{k=1}\frac{1}{Z_k}$
%\begin{align}
%A=\prod^n_{k=1}\frac{1}{Z_k}
%\end{align}
is equal to ${\det T}^{-1}$.% Denoted by $CT(A)=\frac{1}{\det T}$.
\end{proposition}
\begin{proof}
We use induction on n. $n=1$ is clear. %For $n=1$, $T=(T_{11}),A=\frac{1}{T_{11}}$, thus $CT(A)=\frac{1}{\det T}$.
For $1\leq k\leq n$, set $Z^*_{k}=\sum_{1\leq i\leq n-1}T_{ik}\frac{1}{z_i}$, then $$Z_{k}=\left\{
\begin{aligned}
Z^*_{k}z_k+\frac{T_{nk}z_k}{z_n},~~~~& ~~~~1\leq k\leq n-1\\
T_{nn}\left(1+\frac{Z^*_nz_n}{T_{nn}}\right),~~~~&~~~~k=n
\end{aligned}
\right.
$$
so $A$ can be written as
\begin{align}
A=T^{-1}_{nn}\left(Z^*_1z_1+\frac{T_{n1}z_1}{z_n}\right)^{-1}\cdots \left(Z^*_{n-1}z_{n-1}+\frac{T_{nn-1}z_{n-1}}{z_n}\right)^{-1}\left(1+\frac{Z^*_{n}}{T_{nn}}z_n\right)^{-1}.
\end{align}
\par As a function of $z_n$, it follows from Lemma \ref{le1} that CT(A) is
\begin{align}
A^{\prime}=T^{-1}_{nn}\left(Z^{*}_{1}z_1-\frac{T_{n1}z_1Z^{*}_{n}}{T_{nn}}\right)^{-1}\left(Z^{*}_{2}z_2-\frac{T_{n2}z_2Z^{*}_{n}}{T_{nn}}\right)^{-1}\cdots \left(Z^{*}_{n-1}z_{n-1}-\frac{T_{nn-1}z_{n-1}Z^{*}_{n}}{T_{nn}}\right)^{-1}.
\end{align}
Put $T=(T_{ij})_{1\leq i,j\leq n}$ and $T^{\prime}_{ij}=T_{ij}-\frac{T_{in}T_{nj}}{T_{nn}}~(1\leq i,j\leq n-1)$ and $Z^{\prime}_{k}=\sum_{1\leq i\leq n-1}T^{\prime}_{ik}\frac{z_k}{z_i}$, then
\begin{align}
A^{\prime}=T^{-1}_{nn}\prod^{n-1}_{k=1}\frac{1}{Z^{\prime}_k}.
\end{align}
\par Using the induction hypothesis, we have that
\begin{align}
CT(A)=\frac{1}{T_{nn}}CT(A^{\prime})=\frac{1}{T_{nn}}\frac{1}{\det T^{\prime}},
\end{align}
where $T^{\prime}=(T^{\prime}_{ij})_{1\leq i,j\leq n-1}$. Clearly $\det T=T_{nn}\det T^{\prime}$, %by elementary transformations,
we have shown
the proposition.
%\begin{align}
%CT(A)=\frac{1}{\det T}.
%\end{align}
\end{proof}
For a positive integer $M$, let
\begin{align*}
&D_m=\sum_{1\leq n\leq N}s_{(m,1^{n-1})}\{x\}\varphi_{n-1},~1\leq m\leq M, ~\{x\}=\{x_1,\dots,x_N\},\\
&F_m=\sum_{1\leq n\leq N}s_{(m,1^{n-1})}\{y\}\varphi^*_{1-n},1\leq m\leq M, ~\{y\}=\{y_1,\dots,y_N\}.
\end{align*}
Then
\begin{align*}
T_{ij}=[D_i,F_j]=\sum_{1\leq n\leq N}s_{(i,1^{n-1})}\{x\}s_{(j,1^{n-1})}\{y\},%~~I+T=(\delta_{ij}+T_{ij})_{1\leq i,j\leq M},~~Z_{j}=\sum_{1\leq i\leq n}(\delta_{ij}+T_{ij})\frac{z_j}{z_i}.
\end{align*}

Let $T=(T_{ij})_{1\leq i,j\leq M}$, and consider the functions $Z_{j}=\sum_{1\leq i\leq n}(\delta_{ij}+T_{ij})\frac{z_j}{z_i}$.
By (\ref{B3}), (\ref{B2}) and (\ref{bc4}), we have that
\begin{align}
\notag&\langle 0|\mathbf{C}^{*}(x_N)\cdots\mathbf{C}^{*}(x_1)\mathbf{B}^{*}(y_1)\cdots \mathbf{B}^{*}(y_N)|0\rangle\\
\notag=&\langle 0|\sum_{k_m\geq 0}\prod^{M}_{m=1}\frac{\left(D_m(-1)^{m-1}\varphi^*_m\right)^{k_m}}{k_m!}\sum_{l_m\geq 0}\prod^{M}_{m=1}\frac{\left(F_m(-1)^{m-1}\varphi_{-m}\right)^{l_m}}{l_m!}|0\rangle\\
\notag=&\sum_{k_m\geq 0}\langle 0|\prod^{M}_{m=1}\frac{(D_m)^{k_m}}{k_m!}\prod^{M}_{m=1}\frac{(F_m)^{k_m}}{k_m!}|0\rangle(-1)^{\sum^{M}_{m=1} k_m}\prod^{M}_{m=1}k_{m}!\\
\notag=&\sum_{k_m\geq 0}\langle 0|\prod^{M}_{m=1}\frac{(-D_m)^{k_m}}{k_m!}\prod^{M}_{m=1}(F_m)^{k_m}|0\rangle\\
\notag=&CT\left(\langle 0|\prod^{M}_{i=1}\exp\left(-\frac{D_i}{z_i}\right)\prod^{M}_{j=1}\frac{1}{1-z_jF_j}|0\rangle\right)=CT\left(\prod^M_{j=1}\frac{1}{Z_j}\right)=\frac{1}{\det (I+T)}\\
\notag=&\frac{1}{\sum\limits_{\mu\subseteq[N,M]}s_\mu\{x\}s_\mu\{y\}}.
\end{align}

\section{The phase model}

 The vector space $\mathcal{V}$ (resp. dual space $\mathcal{V}^*$)
 is defined as the linear span of all states of the whole lattice with the bases
 \begin{align}
\text{Basis}(\mathcal{V})&=\Big{\{}|\uline{n}\rangle=|n_0\rangle_0\otimes |n_1\rangle_1\otimes\cdots\otimes|n_M\rangle_M; n_i\in\mathbb Z_+\Big{\}},\\
\text{Basis}(\mathcal{V}^*)&=\Big{\{}\langle \uline{m}|=\langle m_0|_0\otimes \langle m_1|_1\otimes\cdots\otimes\langle m_M|_M; m_i\in\mathbb Z_+\Big{\}}.
\end{align}
% where once again $\{m_0,m_1,\dots,m_M\}$ range over all non-negative integers.
 %The action of a basis vector $\langle m|\in \mathcal{V}^*$ on $|n\rangle\in \mathcal{V}$ is given by
% \begin{align}
% \langle m|n\rangle=\prod^{M}_{i=0}\delta_{m_i,n_i}.
% \end{align}
The \textit{phase algebra} \cite{Bog2005,BIK1997} is generated by $\{\pi_i,\psi^\dag_i,\mathcal{N}_i,\psi_i\}_{0\leq i\leq M}$ subject to the following relations:   %and $i$-boson algebras $\{\rho_i,\rho^\dag_i, \tilde{\mathcal{N}}_i\}_{0\leq i\leq M}$}
 \begin{gather}
 [\psi_i,\psi^\dag_j]=\delta_{i,j}\pi_i,~[\mathcal{N}_i,\psi_j]=-\delta_{i,j}\psi_i,~[\mathcal{N}_i,\psi^\dag_j]=\delta_{i,j}\psi^\dag_i,~\psi_i\pi_i=\pi_i\psi^\dag_i=0, ~ 0\leq i,j\leq M.
%{\color{red} [\rho_i,\rho^\dag_j]=\delta_{i,j}(-1)^{\tilde{\mathcal{N}}_i},~[\tilde{\mathcal{N}}_i,\rho_j]=-\delta_{i,j}\rho_i,~[\tilde{\mathcal{N}}_i,\rho^\dag_j]=\delta_{i,j}\rho^\dag_i
% ,~ 0\leq i,j\leq M}
 \end{gather}

 They act on the Fock space $\mathcal V$ naturally as annihilation ($\psi_i$) and creation ($\psi_i^{\dagger}$) operators as follows.
 %The phase algebras possesses the following representations in the spaces $\mathcal{V}$ and $\mathcal{V}^*$
 \begin{align*}
&\psi_i|n_0\rangle_0\otimes\cdots\otimes|n_M\rangle_M=\delta_{n_i-1\geq 0}|n_0\rangle_0\otimes\cdots\otimes|n_i-1\rangle_i\otimes\cdots\otimes|n_M\rangle_M,\\
&\psi^\dag_i|n_0\rangle_0\otimes\cdots\otimes|n_M\rangle_M=|n_0\rangle_0\otimes\cdots\otimes|n_i+1\rangle_i\otimes\cdots\otimes|n_M\rangle_M,\\
&\mathcal{N}_i|n_0\rangle_0\otimes\cdots\otimes|n_M\rangle_M=n_i|n_0\rangle_0\otimes\cdots\otimes|n_M\rangle_M,\\
&\pi_i|n_0\rangle_0\otimes\cdots\otimes|n_M\rangle_M=\delta_{n_i,0}|n_0\rangle_0\otimes\cdots\otimes|n_M\rangle_M,
%&\langle n_0|_0\otimes\cdots\otimes\langle n_M|_M\psi_i=\langle n_0|_0\otimes\cdots\otimes \langle n_i+1|_i\otimes\cdots\otimes\langle n_M|_M,\\
%&\langle n_0|_0\otimes\cdots\otimes\langle n_M|_M\psi^\dag_i=\delta_{n_i-1\geq 0}\langle n_0|_0\otimes\cdots\otimes \langle n_i-1|_i\otimes\cdots\otimes\langle n_M|_M,\\
%&\langle n_0|_0\otimes\cdots\otimes\langle n_M|_M\mathcal{N}_i=n_i\langle n_0|_0\otimes\cdots\otimes\langle n_M|_M,\\
%&\langle n_0|_0\otimes\cdots\otimes\langle n_M|_M\pi_i=\delta_{n_i,0}\langle n_0|_0\otimes\cdots\otimes\langle n_M|_M,
\end{align*}
where $\delta_{n_i-1\geq 0}=1$ if $n_i-1\geq 0$ and $0$ otherwise. Similarly the action on the dual space $\mathcal V^*$ can be defined
accordingly.
The $L$-matrix for the phase model on the space $\mathcal V_1\otimes \mathcal V_2$ ($\mathcal V_i\simeq \mathcal V$)
has the form
\begin{align}
L_{im}(x)=\begin{pmatrix}
x^{-\frac12}&\psi^\dag_m\\
\psi_m&x^{\frac12}
\end{pmatrix},
\end{align}
which satisfies the RLL equation $R(x, y)L_{1m}(x)L_{2m}(y)=L_{2m}(y)L_{1m}(x)R(x, y)$, where the R-matrix is given by
\begin{align}\label{e:Rmat1}
R(x,y)=\begin{pmatrix}
x&0&0&0\\
0&0&x^{\frac12}y^{\frac12}&0\\
0&x^{\frac12}y^{\frac12}&x-y&0\\
0&0&0&x
\end{pmatrix}.
\end{align}

The monodromy matrix $T(x)$ also satisfies the RLL relation and is of the form
\begin{align}
T(x)=L_M(x)\dots L_0(x)=
\begin{pmatrix}
A(x)&B(x)\\
C(x)&D(x)
\end{pmatrix},
\end{align}
where $L_{m}=L_{im}$ ($i=1, 2$) and the off-diagonal operators satisfy
\begin{align}
[B(x),B(y)]=[C(x),C(y)]=0.
\end{align}

Denote $\mathbb{B}(x)=x^{\frac{M}{2}}B(x)$ and $\mathbb{C}(x)=x^{\frac{M}{2}}C(\frac{1}{x})$. Let
$\mathcal{M}_{bc}:\mathcal{V}\rightarrow \mathcal{F}^{(0)}$ and $\mathcal{M}^*_{bc}:\mathcal{V}^*\rightarrow \mathcal{F}^{*(0)}$ be the linear maps defined by $
\mathcal{M}_{bc}(|\uline{n}\rangle)=|\nu), \ \ \mathcal{M}^*_{bc}(\langle \uline{n}|)=(\nu|$,
where $|\nu)=|M^{n_M},\dots,1^{n_1})\in \mathcal{F}^{(0)}$ (cf. \ref{bc14}) and $(\nu|=(M^{n_M},\dots,1^{n_1}|\in \mathcal{F}^{(0)}$(cf. \ref{bc15}). Note that these maps are not one-to-one since they are insensitive to the value of $n_0$.
\begin{proposition}\cite{Bog2005,Whe2012}\label{pro4}
If $\mathcal{M}_{bc}(|\uline{n}\rangle)=|\nu)$, $\mathcal{M}^*_{bc}(\langle \uline{n}|)=(\nu|$, then
\begin{align}
\label{ch4-ph-eq2}\mathcal{M}_{bc}\mathbb{B}(x)|\uline{n}\rangle=\sum_{\nu\prec \mu\subseteq [l+1,M]}x^{|\mu|-|\nu|}|\mu),\\
\label{ch4-ph-eq3}\langle \uline{n}|\mathbb{C}(x)\mathcal{M}^*_{bc}=\sum_{\nu\prec \mu\subseteq [l+1,M]}x^{|\mu|-|\nu|}(\mu|.
\end{align}
\end{proposition}
The following result gives a correspondence between $\mathbf{B}(x), \mathbf{C}(x)$ and $\mathbb B(x), \mathbb C(x)$. It can be seen
from Theorem \ref{th4}, Corollary \ref{cor2} and Proposition \ref{pro4}.

\begin{proposition}
The following commutative diagrams hold.
\begin{displaymath}
\xymatrix{\mathcal{V} \ar[r]^{\mathcal{M}_{bc}} \ar[d]_{\mathbb{B}(x)}& \mathcal{F}^{(0)} \ar[d]^{\mathbf{B}(x)} \\
          \mathcal{V} \ar[r]_{\mathcal{M}_{bc}} & \mathcal{F}^{(0)} }
\xymatrix{\mathcal{V}^* \ar[r]^{\mathcal{M}^*_{bc}} \ar[d]_{\mathbb{C}(x)}& \mathcal{F}^{*(0)} \ar[d]^{\mathbf{C}(x)} \\
          \mathcal{V}^* \ar[r]_{\mathcal{M}^*_{bc}} & \mathcal{F}^{*(0)} }
\end{displaymath}
\end{proposition}
%In addition, $\langle m|n\rangle=(\langle m|\mathcal{M}^*_{bc},\mathcal{M}_{bc}|n\rangle)$ for all $\langle m|\in\mathcal{V}^*,|n\rangle\in\mathcal{V}$ which satisfy $m_0=n_0$, where $(\cdot,\cdot)$ is defined by \ref{bc16}.

%\subsection{Relation between $\mathbf{B}(x),\mathbf{C}(x)$ and $\mathbb{B}(x),\mathbb{C}(x)$}\label{sub3}

%As for the relation between
%%The definition and properties of
%$\mathbb{B}(x)$ or $\mathbb{C}(x)$ and %related to
%the diagonal elements of the monodromy matrix $T(x)$ for the phase model will be discussed Appendix B.

\begin{remark} Using Propositions \ref{pro1}-\ref{pro2} we get another proof of the following result, which is
fundamental in previous study of the phase model in the literature.
\begin{align}
\lim_{M\rightarrow \infty}\mathcal{M}_{bc}\mathbb{B}(x)|\uline{n}\rangle=\exp\left(\sum^{\infty}_{n=1}\frac{x^n}{n}H_{-n}\right)|\nu),
\end{align}
where $\mathcal{M}_{bc}(|\uline{n}\rangle)=|\nu)$.
\end{remark}

\section{The $i$-boson model}
 The vector space $\tilde{\mathcal{V}}$ and its dual $\tilde{\mathcal{V}}^*$ are defined respectively as the linear spans of the bases
 \begin{align}
\text{Basis}(\tilde{\mathcal{V}})&=\Big{\{}|\uline{\tilde{n}}\rangle=|n_0\rangle_0\otimes |n_1\rangle_1\otimes\cdots\otimes|n_M\rangle_M; n_0\geq 0, n_i=0,1\Big{\}},\\
\text{Basis}(\tilde{\mathcal{V}}^*)&=\Big{\{}\langle \uline{\tilde{m}}|=\langle m_0|_0\otimes \langle m_1|_1\otimes\cdots\otimes\langle m_M|_M;
m_0\geq 0, m_i=0,1\Big{\}},
\end{align}

The \textit{$i$-boson algebra} \cite{Whe2012} is generated by $\{\tilde{\psi}^\dag_i,\tilde{\mathcal{N}}_i,\tilde{\psi}_i\}_{0\leq i\leq M}$ satisfying the relations:
\begin{align}
[\tilde{\psi}_i,\tilde{\psi}^\dag_j]=\delta_{i,j}(-1)^{\mathcal{N}_i},~~[\tilde{\mathcal{N}}_i,\tilde{\psi}_j]=-\delta_{i,j}\tilde{\psi}_i,~~[\tilde{\mathcal{N}}_i,\tilde{\psi}^\dag_j]
=\delta_{i,j}\tilde{\psi}^\dag_i
\end{align}
for all $0\leq i,j\leq M$.
The $i$-boson algebra naturally acts on $\tilde{\mathcal{V}}$ by
 \begin{align*}
&\tilde{\psi}_j|n_0\rangle_0\otimes\cdots\otimes|n_M\rangle_M=\frac{1}{\sqrt{2}}\delta_{n_j-1\geq 0}|n_0\rangle_0\otimes\cdots\otimes|n_j-1\rangle_j\otimes\cdots\otimes|n_M\rangle_M,~1\leq j\leq M,\\
&\tilde{\psi}_0|n_0\rangle_0\otimes\cdots\otimes|n_M\rangle_M=\frac{1-(-1)^{n_0}}{\sqrt{2}}|n_0-1\rangle_0\otimes|n_1\rangle_1\otimes\cdots\otimes|n_M\rangle_M,
\\
&\tilde{\psi}^\dag_j|n_0\rangle_0\otimes\cdots\otimes|n_M\rangle_M=\frac{1+(-1)^{n_j}}{\sqrt{2}}|n_0\rangle_0\otimes\cdots\otimes|n_j+1\rangle_i
\otimes\cdots\otimes|n_M\rangle_M,~1\leq j\leq M,\\
&\tilde{\psi}^\dag_0|n_0\rangle_0\otimes\cdots\otimes|n_M\rangle_M=\frac{1}{\sqrt{2}}|n_0+1\rangle_0\otimes\otimes|n_1\rangle_1\otimes\cdots\otimes|n_M\rangle_M,\\
&\tilde{\mathcal{N}}_j|n_0\rangle_0\otimes\cdots\otimes|n_M\rangle_M=n_j|n_0\rangle_0\otimes\cdots\otimes|n_M\rangle_M,~0\leq j\leq M.
%&\pi_i|n_0\rangle_0\otimes\cdots\otimes|n_M\rangle_M=\delta_{n_i,0}|n_0\rangle_0\otimes\cdots\otimes|n_M\rangle_M,\\
%&\langle n_0|_0\otimes\cdots\otimes\langle n_M|_M\tilde{\psi}_j=\frac{1+(-1)^{n_j}}{\sqrt{2}}\langle n_0|_0\otimes\cdots\otimes \langle n_j+1|_j\otimes\cdots\otimes\langle n_M|_M,~1\leq j\leq M,\\
%&\langle n_0|_0\otimes\cdots\otimes\langle n_M|_M\tilde{\psi}_0=\frac{1}{\sqrt{2}}\langle n_0+1|_0\otimes\langle n_1|_1\otimes\cdots\otimes\langle n_M|_M,\\
%&\langle n_0|_0\otimes\cdots\otimes\langle n_M|_M\tilde{\psi}^\dag_j=\frac{1}{\sqrt{2}}\delta_{n_i-1\geq 0}\langle n_0|_0\otimes\cdots\otimes \langle n_j-1|_j\otimes\cdots\otimes\langle n_M|_M,~0\leq j\leq M,\\
%&\langle n_0|_0\otimes\cdots\otimes\langle n_M|_M\tilde{\psi}^\dag_0=\frac{1-(-1)^{n_0}}{\sqrt{2}}\langle n_0-1|_0\otimes\langle n_1|_1\otimes\cdots\otimes\langle n_M|_M,\\
%&\langle n_0|_0\otimes\cdots\otimes\langle n_M|_M\tilde{\mathcal{N}}_j=n_j\langle n_0|_0\otimes\cdots\otimes\langle n_M|_M,~0\leq j\leq M.
%%&\langle n_0|_0\otimes\cdots\otimes\langle n_M|_M\pi_i=\delta_{n_i,0}\langle n_0|_0\otimes\cdots\otimes\langle n_M|_M.
\end{align*}
The action on $\tilde{\mathcal{V}}^*$ is defined similarly.

The $R$-matrix for the $i$-boson model is given by
\begin{align}
\tilde{R}(x,y)=\begin{pmatrix}
x+y&0&0&0\\
0&y-x&2x^{\frac12}y^{\frac12}&0\\
0&2x^{\frac12}y^{\frac12}&x-y&0\\
0&0&0&x+y
\end{pmatrix}.
\end{align}

The RLL equation $\tilde{R}(x, y)\tilde{L}_1(x)\tilde{L}_2(y)=\tilde{L}_2(y)\tilde{L}_1(x)\tilde{R}(x, y)$ on $\tilde{\mathcal V}\otimes \tilde{\mathcal V}$
is satisfied by the $L$-matrix:
\begin{align}
\tilde{L}_{im}(x)=\begin{pmatrix}
x^{-\frac12}&\sqrt{2}\tilde{\psi}^\dag_m\\
\sqrt{2}\tilde{\psi}_m&x^{\frac12}
\end{pmatrix}, \qquad i=1, 2.
\end{align}

The monodromy matrix $\tilde{T}(x)$ has the form (also satisfying the RLL relation)
\begin{align}
\tilde{T}(x)=\tilde{L}_M(x)\dots \tilde{L}_0(x)=
\begin{pmatrix}
\tilde{A}(x)&\tilde{B}(x)\\
\tilde{C}(x)&\tilde{D}(x)
\end{pmatrix},
\end{align}
where $\tilde{L}_{m}=\tilde{L}_{im}$ and the off-diagonal operators satisfy
\begin{align}
[\tilde{B}(x),\tilde{B}(y)]=[\tilde{C}(x),\tilde{C}(y)]=0.
\end{align}

Denote $\tilde{\mathbb{B}}(x)=x^{\frac{M}{2}}\tilde{B}(x)$ and $\tilde{\mathbb{C}}(x)=x^{\frac{M}{2}}\tilde{C}(\frac{1}{x})$. Let $\mathcal{M}_\phi:\tilde{\mathcal{V}}\rightarrow \mathcal{F}_\phi^{(0)}$ and $\mathcal{M}^*_\phi:\tilde{\mathcal{V}}^*\rightarrow \mathcal{F}_\phi^{*(0)}$ be the two linear maps defined by
\begin{align}
\mathcal{M}_\phi(|\tilde{\uline{n}}\rangle)=2^{-l(\tilde{\nu})}|\tilde{\nu}),~~~~~\mathcal{M}^*_\phi(\langle \tilde{\uline{n}}|)=2^{-l(\tilde{\nu})}(\tilde{\nu}|,
\end{align}
where $|\tilde{\nu})=|M^{n_M},\dots,1^{n_1})$ if $\sum_in_i$ is even and $|\tilde{\nu})=|M^{n_M},\dots,1^{n_1},0)$ if $\sum_in_i$ is odd.
The dual vectors $(\tilde{\nu}|$ is defined similarly.
\begin{proposition}\cite{Bog2005,Whe2012}\label{pro5}
If $\mathcal{M}_\phi(|\tilde{\uline{n}}\rangle)=|\tilde{\nu})$, $\mathcal{M}^*_\phi(\langle \tilde{\uline{n}}|)=(\tilde{\nu}|$, then
\begin{align}
\label{ch4-ph-eq2b}\mathcal{M}_\phi\tilde{\mathbb{B}}(x)|\tilde{\uline{n}}\rangle=\sum_{\tilde{\nu}\prec \tilde{\mu}}2^{\#(\tilde{\mu}|\tilde{\nu})-l(\tilde{\mu})+l(\tilde{\nu})}x^{|\tilde{\mu}|-|\tilde{\nu}|}|\tilde{\mu}),\\
\label{ch4-ph-eq3b}\langle \tilde{\uline{n}}|\tilde{\mathbb{C}}(x)\mathcal{M}^*_\phi=\sum_{\tilde{\nu}\prec \tilde{\mu}}2^{\#(\tilde{\mu}|\tilde{\nu})-l(\tilde{\mu})+l(\tilde{\nu})}x^{|\tilde{\mu}|-|\tilde{\nu}|}(\tilde{\mu}|.
\end{align}
\end{proposition}

The following can be seen from Theorem \ref{th5} and Proposition \ref{pro5}.
 \begin{proposition}
 %Let $\Sigma(\uline{\tilde{n}})=\sum^M_{j=1}n_j$.
The following commutative diagrams hold.
\begin{displaymath}
\xymatrix{\tilde{\mathcal{V}} \ar[r]^{\mathcal{M}_\phi} \ar[d]_{\tilde{\mathbb{B}}(x)}& \tilde{\mathcal{F}}^{(0)} \ar[d]^{\tilde{\mathbf{B}}(x)} \\
          \tilde{\mathcal{V}} \ar[r]_{\mathcal{M}_\phi} & \tilde{\mathcal{F}}^{(0)} }
\xymatrix{\tilde{\mathcal{V}}^* \ar[r]^{\mathcal{M}^*_\phi} \ar[d]_{\tilde{\mathbb{C}}(x)}& \tilde{\mathcal{F}}^{*(0)} \ar[d]^{\tilde{\mathbf{C}}(x)} \\
          \tilde{\mathcal{V}}^* \ar[r]_{\mathcal{M}^*_\phi} & \tilde{\mathcal{F}}^{*(0)} }
\end{displaymath}
\end{proposition}
\end{appendix}

\end{document}